\documentclass[10pt]{llncs}
\usepackage{hyperref}
\usepackage{epsfig}
\usepackage{graphicx}
\usepackage{amsfonts}
\usepackage{amssymb}
\usepackage{amsmath}
\usepackage{latexsym}
\usepackage{enumerate}
\usepackage{fullpage} 
\usepackage[linesnumbered,ruled]{algorithm2e}

\def\bsquareforqed{\rule{0.6em}{0.6em}}
\def\bqed{\ifmmode\bsquareforqed\else{\unskip\nobreak\hfil
\penalty50\hskip1em\null\nobreak\hfil\bsquareforqed
\parfillskip=0pt\finalhyphendemerits=0\endgraf}\fi}
\spnewtheorem{observation}{Observation}{\bfseries}{\itshape}
\spnewtheorem{condition}{Condition}{\bfseries}{\itshape}
\spnewtheorem{clm}{Claim}{\bfseries}{\rmfamily}
\newcommand{\conf}[1]{} 
\newcommand{\confExt}[1]{#1} 
\newcommand{\journ}[1]{#1} 

\newcommand{\ds}{\displaystyle}  
\newcommand{\comp}[1]{\overline{#1}}  

\newcommand{\boxi}{\mbox{\textnormal{box}}}

\newcommand{\tw}{\hspace{0.5mm}\mbox{\textnormal{tree-width}\hspace{.5mm}}}

\newcommand{\floor}[1]{\left\lfloor #1 \right\rfloor}

\newcommand{\mA}{\mathcal{A}}
\newcommand{\mB}{\mathcal{B}}
\newcommand{\mC}{\mathcal{C}}

\newcommand{\mN}{\mathcal{N}}
\newcommand{\mP}{\mathcal{P}}
\newcommand{\mR}{\mathcal{R}}
\newcommand{\mS}{\mathcal{S}}

\makeatletter
\renewcommand\subsubsection{\@startsection{subsubsection}{3}{\z@}%
                       {-18\p@ \@plus -4\p@ \@minus -4\p@}%
                       {8\p@ \@plus 4\p@ \@minus 4\p@}%
                       {\normalfont\normalsize\bfseries\boldmath
                        \rightskip=\z@ \@plus 8em\pretolerance=10000 }}
\makeatletter

\begin{document}

\title{Representing a cubic graph as the intersection graph of
axis-parallel boxes in three dimensions}
\author{Abhijin Adiga\inst{1}, L. Sunil Chandran\inst{1}}
\institute{Department of Computer Science and Automation, Indian Institute
of Science, Bangalore--560012, India. \\email: \{abhijin,sunil\}@csa.iisc.ernet.in}
\date{}
\maketitle
\begin{abstract}
We show that every graph of maximum degree $3$ can be represented as the
intersection graph of axis parallel boxes in three dimensions, that is,
every vertex can be mapped to an axis parallel box such that two boxes
intersect if and only if their corresponding vertices are adjacent. In
fact, we construct a representation in which any two intersecting boxes
just touch at their boundaries. \confExt{Further, this construction can be
realized in linear time.}

\keywordname{ cubic graphs, intersection graphs, axis parallel boxes,
boxicity}
\end{abstract}
\section{Introduction}
We will be considering only simple, undirected and finite graphs.
Let ${\cal F}=\{S_1, S_2, \ldots, S_n\}$ be a family of sets. An
intersection graph associated with ${\cal F}$ has ${\cal F}$ as the
vertex set and two vertices $S_i$ and $S_j$ are adjacent if and only if
$i\neq j$ and $S_i\cap S_j\neq \emptyset$. It is interesting to study
intersection graphs of sets with some restriction, for example,
sets which correspond to geometric objects such as intervals, spheres,
boxes, axis-parallel lines, etc. Many important graph classes arise out
of such restrictions: interval graphs, circular arc graphs, unit-disk
graphs and grid-intersection graphs, to name a few. In this paper, we are
concerned with intersection graphs of $3$-dimensional boxes. A {\it
$3$-dimensional axis parallel box} ($3$-box in short) is a Cartesian
product of $3$ closed intervals on the real line. A graph is said to
have a {\it $3$-box representation} if it can be represented as the
intersection graph of $3$-boxes.

In the literature there are several results on representing a planar
graph as the intersection graph of various geometric objects. Among
these, the most noted result is the circle packing theorem (also known
as the Koebe-Andreev-Thurston theorem) from which it follows that
planar graphs are exactly the intersection graphs of closed disks in
the plane such that the intersections happen only at the boundaries.
In \cite{intervalRepPlanarGraphsThomassen}, Thomassen gave a similar
representation for planar graphs with $3$-boxes. He showed that every
planar graph has a {\it strict $3$-box representation}, that is,
intersections occur only in the boundaries of the boxes and two boxes
which intersect have precisely a $2$-box (a rectangle) in common.  Very
recently, Felsner and Francis \cite{contactRepPlanarGraphsFelsnerFrancis}
strengthened this result by showing that there exists a strict $3$-box
representation for a planar graph such that each box is an isothetic cube. In
\cite{phdThesisScheinerman,intervalRepPlanarGraphsThomassen},
it was shown that every planar graph has a strict representation
using at most two rectangles per vertex. Scheinerman and West
\cite{intervalNumberPlanarScheinermanWest} showed that every planar
graph is an intersection graph of intervals such that each vertex is
represented by at most three intervals on the real line.

We consider the question of whether a graph of maximum degree $3$
has a $3$-box representation. We note that there exist graphs
with maximum degree greater than $3$ which do have a $3$-box
representation. For example, it is easy to show that a $K_8$
minus a perfect matching does not have a $3$-box representation
\cite{recentProgressesInCombRoberts}. Considering the effort that has
gone into discovering geometric representation theorems for planar graphs,
it is surprising that no such results are known up to now in the case of
cubic graphs. It may be because of the fact that intuitively cubic graphs
are farther away from ``geometry" compared to planar graphs. In this paper
we present the first such theorem (as far as we know) for cubic graphs:
\begin{theorem}\label{thm:mainTheorem}
Every graph of maximum degree $3$ has a $3$-box representation with the
restriction that two boxes can intersect only at their boundaries.
\end{theorem}

\subsection{$k$-box representations and boxicity}
The concept of $3$-box representation can be extended to higher
dimensions. A $k$-box is a Cartesian product of closed intervals
$[a_1,b_1]\times [a_2,b_2]\times\cdots\times [a_k,b_k]$. A graph $G$
has a {\it $k$-box representation} if it is the intersection graph of
a family of $k$-boxes in the $k$-dimensional Euclidean space. The
\emph{boxicity} of $G$ denoted by $\boxi(G)$, is the minimum
integer $k$ such that $G$ has a $k$-box representation. Clearly,
Theorem \ref{thm:mainTheorem} can be rephrased as: Every graph with
maximum degree $3$ has boxicity at most $3$.  The best known upper
bound for the boxicity of cubic graphs is $10$; it follows from
the bound $\boxi(G)\le 2\floor{\frac{\Delta^2}{2}}+2$ by Esperet
\cite{boxicityGraphsBoundedDegreeEsperet}, where $\Delta$ is the maximum
degree of the graph. In \cite{boxicityMaxDegreeSunilNaveenFrancis}, it
was conjectured that boxicity of a graph is $O(\Delta)$. However, this
was disproved in \cite{boxicityPosetDimensionAbhijinSunilDiptendu} by
showing the existence of graphs with boxicity $\Omega(\Delta\log\Delta)$.
Theorem \ref{thm:mainTheorem} implies that the conjecture is true for
$\Delta=3$.

Our result also implies that any problem which is hard for cubic graphs
is also hard for graphs with a $3$-box representation. We list a few
of such problems: crossing number, minimum vertex cover, Hamiltonian cycle,
maximum independent set, minimum dominating set and maximum cut.

We give a brief literature survey on boxicity. It was introduced
by Roberts in 1969 \cite{recentProgressesInCombRoberts}. Cozzens
\cite{phdThesisCozzens} showed that computing the boxicity of
a graph is NP-hard. This was later strengthened by Yannakakis
\cite{complexityPartialOrderDimnYannakakis} and finally by
Kratochv\`{\i}l \cite{specialPlanarSatisfiabilityProbNPKratochvil}
who showed that determining whether boxicity of a graph
is at most two itself is NP-complete. Adiga, Bhowmick and
Chandran \cite{boxicityPosetDimensionAbhijinSunilDiptendu}
showed that it is hard to approximate the boxicity of even
a bipartite graph within $\sqrt{n}$ factor, where $n$ is the
order of the graph. In \cite{computingBoxicityCozzensRoberts},
Cozzens and Roberts studied the boxicity of split graphs.
Chandran and Sivadasan \cite{boxicityTreewidthSunilNaveen}
showed that  $\boxi(G)\le\tw(G)+2$. Chandran, Francis and
Sivadasan~\cite{boxicityMaxDegreeSunilNaveenFrancis} proved that
$\boxi(G)\le2\chi(G^2)$, where $\chi$ is the chromatic number and $G^2$
is the square of the graph. 

Boxicity is a direct generalization of the concept of \emph{interval
graphs}. A graph is an interval graph if and only if it can be expressed
as the intersection of a family of intervals on the real line.
Since a $1$-box is an interval, it follows that interval graphs are
precisely the class of graphs with boxicity at most $1$ \footnote{The
only graph with boxicity $0$ is the complete graph.}. Now we present an
alternate characterization of $k$-box representation in terms of interval
graphs. This is used more frequently than its geometric definition.
\begin{lemma}\label{lem:intbox} 
A graph $G$ has a $k$-box representation if and only if there exist
$k$ interval graphs $I_1,I_2,\ldots,I_k$ such that $V(I_i)=V(G)$,
$i=1,2,\ldots,k$ and $E(I_1)\cap E(I_2)\cap\cdots\cap E(I_k)=E(G)$.
\end{lemma}
Our proof of Theorem \ref{thm:mainTheorem} uses Lemma \ref{lem:intbox}; we
construct $3$ interval graphs such that the given cubic graph is the
intersection of these interval graphs.

We observe that there exist graphs with maximum degree $3$ (and hence
cubic graphs) with boxicity strictly greater than $2$. For example,
Let $G$ be a non-planar cubic graph and $G_s$ be the graph obtained by
subdividing each edge once. Then, $\boxi(G_s)>2$. It is an easy exercise
to prove this. One way is to show that if $G_s$ does have a $2$-box
representation, then a planar embedding for $G$ can be derived from
this box representation, contrary to the initial assumption that $G$ is
a non-planar graph. This means that these graphs do not have a $2$-box
representation, that is, they cannot be expressed as the intersection
graphs of rectangles on the plane. The rest of the paper is devoted to
the proof of Theorem \ref{thm:mainTheorem}.

\subsection{Notation}
Let $G$ be a graph. The notation $(x,y)\in E(G)$ ($(x,y)\notin
E(G)$) means that $x$ is (not) adjacent to $y$ in $G$. For $U\subseteq
V(G)$, $G[U]$ denotes the graph induced by $U$ in $G$. The \emph{open
neighborhood} of $U$, denoted by $N(U,G)$ is the set $\left\{x\in
V(G)\setminus U|\exists y\in U \mbox{ such that }(x,y)\in E(G)\right\}$.
The length of a path is the number of edges in the path. We consider an
isolated vertex as a path of length $0$. Suppose $G$ and $H$ are graphs
defined on the same vertex set. $G\cap H$ denotes the graph with $V(G\cap
H)=V(G)$ ($=V(H)$) and $E(G\cap H)=E(G)\cap E(H)$.

Consider a non-empty set $X$ and let $\Pi$ be an ordering of the elements
of $X$. $\comp{\Pi}$ denotes the reverse of $\Pi$, that is, for any $x,y\in
X$, $\comp{\Pi}(x)<\comp{\Pi}(y)$ if and only if $\Pi(x)>\Pi(y)$. Let
$A$ and $B$ be disjoint subsets of $X$. The notation $\Pi(A)<\Pi(B)$
implies the following: $\forall a\in A,\ b\in B$, $\Pi(a)<\Pi(b)$.

\begin{lemma}\label{lem:cubicEnough}
If every cubic graph has a $3$-box representation, then, every graph of
maximum degree $3$ also has a $3$-box representation. The statement holds
even when the intersections are restricted to the boundaries of the boxes.
\end{lemma}
\journ{
\begin{proof}
Let $H$ be a non-cubic graph with maximum degree $3$. We will show
that there exists a cubic graph $H'$ such that $H$ is an induced
subgraph of $H'$. Here is one way of constructing $H'$ from $H$. Let
$D=3|V(H)|-\sum_{v\in V(H)}d(v)$, where $d(v)$ is the degree of $v$. Let
$C_D$ be a $D$-length cycle such that $V(C_D)\cap V(H)=\varnothing$. We
construct $H'$ as follows: Let $V(H')=V(H)\cup V(C_D)$. $H'$ contains
all the edges contained in $H$ and $C_D$ and in addition, each vertex
$v\in V(H)$ is made adjacent to $3-d(v)$ unique vertices from $V(C_D)$,
where $d(v)$ is the degree of $v$. Clearly, $H'$ is cubic and by
Theorem \ref{thm:mainTheorem}, has a $3$-box representation. Since $H$
is an induced subgraph of $H'$, any box representation of $H'$ can be
converted to a box representation of $H$ by simply retaining only the
boxes of vertices belonging to $H$. \qed
\end{proof}
}
In view of Lemma \ref{lem:cubicEnough}, we note that it is enough to prove
that a cubic graph has a $3$-box representation. Therefore, in our proof of
Theorem \ref{thm:mainTheorem}, we will assume that the graph is cubic.

\section{Structural prerequisites}\label{sec:structPre}
\subsection{Special cycles and paths}
\begin{definition}{\bf Special cycle:}\label{def:specialCycle}
An induced cycle $C$ is a special cycle if for all $x\in C$,
$C\setminus\{x\}$ is not a subgraph of an induced cycle or path of
size $\ge|C|+1$.
\end{definition}
\begin{definition}{\bf Special path:}\label{def:specialPath}
An induced path $P$ is a special path if
\begin{enumerate}
\item it is maximal in the sense that it is not a subgraph of an induced
cycle or a longer induced path, and
\item for any end point of $P$, say $x$, $P\setminus\{x\}$ is not a
subgraph of an induced cycle of size $\ge|P|$ or an induced path of length
$\ge|P|+1$.
\end{enumerate}
\end{definition}

\begin{observation}\label{obs:nonSpecial}
Any connected graph with at least $3$ vertices contains a special cycle
or path.
\end{observation}
\journ{
This is easy to see. Among all sets of vertices which induce cycles or
paths in the graph, consider the largest sets. If one of these sets
induces a cycle, then, clearly this is a special cycle since there is
no larger induced path or cycle in the graph. If none of them induces a cycle,
then, each of these sets induces a special path since there is no induced
longer path or an induced cycle of the same size in the graph.
}
\subsection{Partitioning the vertex set of a cubic graph}
Let $G$ be a cubic graph and let $V=V(G)$. We partition $V$ in
two stages. In Algorithm \ref{alg:primPart}, we obtain the primary
partition: $V=\mS\uplus\mN_1\uplus\mA_1$. This is followed by a finer
partitioning in Algorithm \ref{alg:finePart}: $\mN_1=\mR\uplus\mN$
and $\mA_1=\mB\uplus\mA$.

\begin{algorithm}[h]
\caption{\label{alg:primPart}}
\SetKwInOut{Input}{input}
\SetKwInOut{Output}{output}
\Input{Cubic graph $G$}
\Output{$\mS,\mN_1,\mA_1$ such that $V=\mS\uplus\mN_1\uplus\mA_1$.}
Let $V'=V$ and $\mS=\varnothing$\; 
\While{there is a connected component in $G[V']$ with at least $3$ vertices}{\nllabel{lin:specialContra}
   Let $T\subseteq V'$ be a set which induces a special cycle or path; \tcp{which exists by Observation \ref{obs:nonSpecial}}
   $\mS\longleftarrow \mS\cup T$\;
   $V'\longleftarrow V'\setminus\{T\cup N(T,G[V'])\}$\nllabel{lin:removeNeighbors}\;
}
$\mA_1=V'$\;
$\mN_1=N(\mS,G)$;
\end{algorithm}

\begin{observation}\label{obs:primPart}
We have some easy observations from Algorithm \ref{alg:primPart}:
\begin{enumerate}
\item $\mS$ induces a collection of cycles and paths in $G$.\label{primPart:inducedCyclesPaths}
\item Every vertex in $\mS$ has at least one neighbor in
$\mN_1$.\label{primPart:SN1} Therefore, every vertex in $\mN_1$ is
adjacent to at most two vertices in $\mA_1$.
\item For any $u\in\mS$ and $v\in\mA_1$, $u$ and $v$ are not
adjacent. \label{primPart:ADisjointS}
\item $\mA_1$ induces a collection of isolated vertices and edges
in $G$. \journ{This observation follows from the fact that $G[\mA_1]$ does not
contain any special cycle or path and therefore, from Observation
\ref{obs:nonSpecial}, does not contain any component with three or more
vertices.} \label{primPart:isolatedA}
\end{enumerate}
\end{observation}

\begin{algorithm}[h]
\caption{\label{alg:finePart}}
\SetKwInOut{Input}{input}
\SetKwInOut{Output}{output}
\Input{Cubic graph $G$, $\mN_1$, $\mA_1$}
\Output{$\mN,\mA,\mR,\mB$ such that $\mN_1=\mR\uplus\mN$ and $\mA_1=\mB\uplus\mA$.}
Let $\mR=\mB=\varnothing$\;
\ForEach{$v\in\mN_1$}{
   \tcp{Recall that $v$ is adjacent to at least one vertex in $\mS$ and therefore to at most $2$ vertices in $\mA_1$.}
   \If {$v$ is adjacent to two vertices $\mA_1\setminus\mB$,}{\nllabel{lin:N2A}
      Let $X_1(v)=N(v,G[\mA_1\setminus\mB])$; \tcp{the two neighbors of $v$
      in $\mA_1\setminus\mB$}\nllabel{lin:X1}
      Let $X_2(v)=N(X_1(v),G[\mA_1\setminus\mB])$; \tcp{neighbors of
      neighbors of $v$ in $\mA_1\setminus\mB$}\nllabel{lin:X2} 
      $\mR\longleftarrow\mR\cup\{v\}$\;\nllabel{lin:vR}
      $\mB\longleftarrow\mB\cup(X_1(v)\cup X_2(v))$\;\nllabel{lin:BX1X2}
   }
}
$\mN=\mN_1\setminus\mR$\;
$\mA=\mA_1\setminus\mB$;
\end{algorithm}

\begin{observation}\label{obs:finePart}
Some observations from Algorithm \ref{alg:finePart}.
\begin{enumerate}
\item Every vertex in $\mN$ is adjacent to at most one vertex in
$\mA$.\label{finePart:N1A}
\item Every vertex in $\mR$ is adjacent to one vertex in $\mS$ and
two vertices in $\mB$. This immediately implies that (a) $\mR$ is an
independent set and (b) for any $u\in\mR$ and $v\in\mN\cup\mA$, $u$
and $v$ are not adjacent.\label{finePart:RDisjointNA}
\item Since $\mB\subseteq\mA_1$, for any $u\in\mB$ and $v\in\mS$, $u$
and $v$ are not adjacent, by Observation
\ref{obs:primPart}.\ref{primPart:ADisjointS}.\label{finePart:BDisjointS}
\item For any $u\in\mB$ and $v\in\mA$, $u$ and $v$ are not adjacent.
\journ{The
proof is as follows. Since $u\in\mB$, it follows that there exists a
$w\in\mR$ such that in Algorithm \ref{alg:finePart}, $u\in X_1(w)\cup
X_2(w)$. From Observation \ref{obs:primPart}.\ref{primPart:isolatedA},
$u$ is adjacent to at most one vertex in $\mA_1$. If it does have a
neighbor in $\mA_1$, it must belong to $X_1(w)\cup X_2(w)$ which is a
subset of $\mB$. Since $\mA=\mA_1\setminus\mB$, $u$ is not adjacent to
any vertex in $\mA$.} \label{finePart:BDisjointA}
\end{enumerate}
\end{observation}
\journ{
\begin{observation}\label{obs:X1X2}
We have some observations regarding $X_1(\cdot)$ and $X_2(\cdot)$ which
are defined in Algorithm \ref{alg:finePart}. Let $v\in\mR$.
\begin{enumerate}
\item $|X_1(v)|=2$.
\item Since $X_1(v)\cup X_2(v)\subseteq\mA_1$, from Observation
\ref{obs:primPart}.\ref{primPart:isolatedA} it follows that every vertex
in $X_1(v)\cup X_2(v)$ has at most one neighbor in $\mA_1$ and this
neighbor is in $X_1(v)\cup X_2(v)$.\label{X1X2:closed}
\item If the two vertices in $X_1(v)$ are adjacent, then, $X_2(v)$ is empty.
\item If $X_2(v)$ is not empty, then, again from Observation
\ref{obs:primPart}.\ref{primPart:isolatedA}, every vertex in $X_2(v)$
is adjacent to exactly one vertex in $X_1(v)$ and $|X_2(v)|\le2$.
\end{enumerate}
\end{observation}
}
\subsubsection{Partitioning $\mS$:}\label{sec:partS}
We partition $\mS$ into $\mC$, the set of vertices which induce special
cycles and $\mP$, the set of vertices which induce special paths. $\mP$
is further partitioned into $\mP_e$, the set of end points and $\mP_i$,
the set of interior points of all the paths.
\begin{definition}{\bf Second end points}\label{def:secondEndPoint}
of a path of length at least $2$ are the interior vertices of the path
which are adjacent to at least one of its end points.
\end{definition}
The set of second end points of the paths in $\mP$ is denoted by
$\mP_{2e}$ and $\mP_{2i}=\mP_i\setminus\mP_{2e}$.

\begin{lemma}\label{lem:NCPi}
Every vertex in $\mN_1$ is adjacent to at least one vertex in $\mC\cup\mP_i$.
\end{lemma}
\journ{
\begin{proof}
Let $v\in\mN_1$. Since $\mN_1=N(\mS,G)$, $v$ is adjacent to at least one
vertex in $\mS$. If $v$ is not adjacent to any vertex in $\mC\cup\mP_i$,
it implies that all its neighbors in $\mS$ belong to $\mP_e$. Let $P$ be
the first path in $\mP$ to be extracted in Algorithm \ref{alg:primPart}
with $v$ as its neighbor. Since $v$ is not adjacent to any vertex in
$\mP_i$, it follows that it is adjacent to at least one end point of $P$
and none of its interior vertices. Since $P$ is the first component
of $\mS$ to be extracted with $v$ as the neighbor, it means that $v$
is still present in $V'$ when $P$ is chosen as the special path. If $v$
is adjacent to both the end points of $P$, then $P\cup\{v\}$ induces
a cycle in $G[V']$ and if $v$ is adjacent to only one end point, then
$P\cup\{v\}$ induces a path in $G[V']$ at that stage. In either case we
have a contradiction to the fact that $P$ is a special path of $G[V']$.
\qed
\end{proof}
}

\subsection{The graph induced by $\mS\cup\mR\cup\mB$}\label{sec:SRB}
\begin{lemma}\label{lem:componentsY}
For each $u\in\mR$, let $\Gamma(u)=\{u\}\cup X_1(u)\cup X_2(u)$, where
$X_1(\cdot)$ and $X_2(\cdot)$ are as defined in Algorithm~\ref{alg:finePart}. Then,
\begin{enumerate}
\item $\mR\cup\mB=\ds\biguplus_{u\in\mR}\Gamma(u)$,
\item $\Gamma(u)$ is a component in the graph induced by $\mR\cup\mB$, and
\item $\Gamma(u)$ is isomorphic to one of the graphs illustrated in
Figure~\ref{fig:componentsY}.
\end{enumerate}
\end{lemma}
\journ{
\begin{proof}
From Algorithm \ref{alg:finePart}, it is clear that
$\mR\cup\mB=\bigcup_{u\in\mR}\Gamma(u)$. Therefore, to prove the first
statement we need to only show that for two distinct vertices $u,v\in\mR$,
$\Gamma(u)\cap\Gamma(v)=\varnothing$. Let us assume that $u$ was added
to $\mR$ before $v$ in the algorithm. Since any $x\in X_1(u)\cup X_2(u)$
is present in $\mB$ when $v$ is being added to $\mR$, it implies that
$x\notin X_1(v)\cup X_2(v)$. Hence proved.

Note that $\Gamma(u)$ is connected. We will show that no vertex in
$\Gamma(u)$ is adjacent to any vertex in $\Gamma(v)$, for any $v$ which
was added to $\mR$ after $u$ in Algorithm \ref{alg:finePart}. Clearly,
this will imply that $\Gamma(u)$ is a component in the graph induced
by $\mR\cup\mB$. First, let us consider $u$. Since $u$ is adjacent to
one vertex in $\mS$, it has only two neighbors in $\mA_1$ and these two
neighbors are in $X_1(u)$. This implies that it is not adjacent to any
vertex in $\Gamma(v)$ since $\Gamma(u)\cap\Gamma(v)=\varnothing$.

Consider any vertex $x\in X_1(u)\cup X_2(u)$. From
Observation \ref{obs:X1X2}.\ref{X1X2:closed} and the fact that
$\Gamma(u)\cap\Gamma(v)=\varnothing$, we can infer that $x$ is not
adjacent to any vertex in $X_1(v)\cup X_2(v)$. Now, suppose $x$ is
adjacent to $v$. Since $x\in\mB$ when $v$ is being added to $\mR$, we infer
that $v$ is adjacent to at most one vertex in $\mA_1\setminus\mB$
at that stage. This implies that $v$ does not satisfy the condition in
Line \ref{lin:N2A} in the algorithm, a contradiction since $v$ belongs
to $\mR$. Hence proved.

From Observation \ref{obs:X1X2}, it is easy to infer that each component
$\Gamma(u)$ is isomorphic to one of the five graphs shown in Figure
\ref{fig:componentsY}. Let $X_1(u)=\{x_1,x_1'\}$. If $x_1$
or $x_1'$ has a neighbor in $X_2(u)$, then, it will be denoted by $x_2$
or $x_2'$ respectively. We have the following five graphs: In each graph,
$u$ is adjacent to $x_1$ and $x_1'$.
\begin{enumerate}[(a)]
\item $x_1$ is adjacent to $x_1'$, and therefore, $X_2(u)=\varnothing$.
\item $x_1$ is not adjacent to $x_1'$ and $X_2(u)=\varnothing$.
\item $x_1$ is not adjacent to $x_1'$ and $X_2(u)=\{x_2\}$.
\item $x_1$ is not adjacent to $x_1'$ and $X_2(u)=\{x_2'\}$.
\item $x_1$ is not adjacent to $x_1'$ and $X_2(u)=\{x_2,x_2'\}$.
\qed
\end{enumerate}
\end{proof}
}

\begin{figure}[tb]
\begin{center}
\epsfig{file=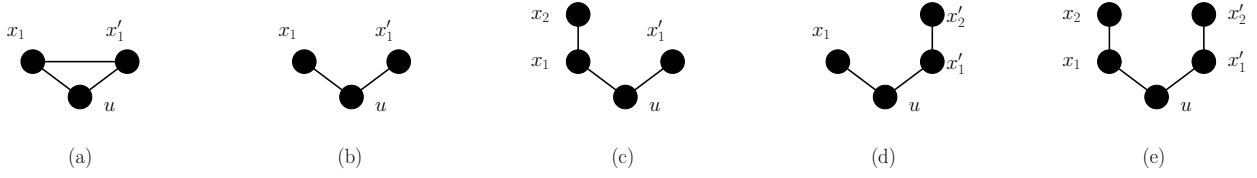,width=16.5cm}
\end{center}
\caption{In the graph induced by $\mR\cup\mB$, each component $\Gamma(u)$,
$u\in\mR$ is isomorphic to one of the graphs illustrated in the figure.
Here, $u$ has exactly two neighbors in $\mB$, $x_1$ and $x_1'$. These
neighbors if not adjacent can each have at most one neighbor in $\mB$
which are denoted by $x_2$ and $x_2'$ respectively.
\label{fig:componentsY}}
\end{figure}

\journ{
From Lemma \ref{lem:componentsY} and Observations
\ref{obs:finePart}.\ref{finePart:BDisjointS} and
\ref{obs:finePart}.\ref{finePart:BDisjointA}, it follows that every
vertex in $\mB$ is adjacent to either one or two vertices in $\mR\cup\mB$
and no vertex in $\mS\cup\mA$. This implies that every vertex in $\mB$
is adjacent to either one or two vertices in $\mN$. Based on this,
we partition $\mB$ into two parts:
}
\begin{definition}{\bf $\mB_1$ and $\mB_2$:}\label{def:B1B2} $\mB_1$
is the set of vertices of $\mB$ which have one neighbor in $\mN$ and
$\mB_2$ is the set of vertices of $\mB$ which have two neighbors in $\mN$.
\end{definition}
\journ{
Recall that from Observation
\ref{obs:finePart}.\ref{finePart:RDisjointNA}, each vertex of $\mR$
has a unique neighbor in $\mS$. In fact, we can infer more:
}
\begin{lemma}\label{lem:RP2i}
Let $u\in\mR$. The unique vertex of $\mS$ to which $u$ is adjacent to
belongs to $\mP_{2i}$.
\end{lemma}
\journ{
\begin{proof}
Let $x$ be the unique neighbor of $u$ in $\mS$. Let $a$ and $b$ be the
remaining neighbors of $u$. From Observation
\ref{obs:finePart}.\ref{finePart:RDisjointNA}, $a,b\in\mB$. We need
to show that $x\in\mP_{2i}$. We will prove by contradiction. Let $T$
be the special cycle or path in $\mS$ which contains $x$. Since $T$
is the only component in $\mS$ with $v$ as a neighbor, it implies that
in Algorithm \ref{alg:primPart}, $v$ is in $V'$ when $T$ is being chosen
as the special path or cycle. Since $a,b\in\mB$, it implies that they
belong to $\mA_1$ and therefore, they too are present in $V'$ when $T$ is
being chosen. Moreover, $a$ and $b$ are not adjacent to any vertex in $T$
(Observation \ref{obs:finePart}.\ref{finePart:BDisjointS}). Now, we have
the following cases to consider:
\begin{description}
\item[$x\in\mC$:] This implies that $T$ is a special cycle. Let $x'\in T$
be a vertex adjacent to $x$. Clearly, $(T\setminus\{x'\})\cup\{v,a\}$
induces a path of length $|T|+1$ in $G[V']$ contradicting the fact that
$T$ is a special cycle.
\item[$x\in\mP_e$:] This implies that $T$ is a special path. Since $v$
is not adjacent to any other vertex in the path, $T\cup\{v\}$ induces
a path of length $|T|+1$ in $G[V']$, contradicting its maximality and thus
$T$ cannot be a special path.
\item[$x\in\mP_{2e}$:] Again, this implies that $T$ is a special
path. Let $x_e$ be an end point of $T$ to which $x$ is adjacent
to. Clearly, $(T\setminus\{x_e\})\cup\{v,a\}$ induces a path of length
$|T|+1$ in $G[V']$, contradicting the fact that $T$ is a special path.
\end{description}
Therefore, $x\in\mP_{2i}$.  \qed
\end{proof}
}
\begin{observation}\label{obs:RP2i}
We have the following observations due to Lemma \ref{lem:RP2i}.
\begin{enumerate}
\item Each vertex in $\mC\cup\mP_{2e}$ is adjacent to exactly one
vertex in $\mN$. \journ{The proof is as follows: Note that each
vertex in $\mC\cup\mP_{2e}$ either belongs to a special cycle or is an
interior vertex of a special path in $G[\mC\cup\mP]$. Therefore, it
has only one neighbor in $V\setminus(\mC\cup\mP)$ and by Observation
\ref{obs:primPart}.\ref{primPart:SN1}, it must belong to $\mN_1$. By
Lemma \ref{lem:RP2i}, it does not belong to $\mR$. Hence, it belongs
to $\mN$.} \label{RP2i:CP2eDisjointR}
\item Every vertex in $\mP_e$ is adjacent to exactly two vertices in $\mN$.
\label{RP2i:Pe2N}
\item Every vertex in $\mP_{2i}$ is adjacent to exactly one vertex in
$\mR\cup\mN=\mN_1$. \label{RP2i:P2iRN}
\end{enumerate}
\end{observation}

\subsection{The graph induced by $\mB_2\cup\mP_e\cup\mN$}\label{sec:B2PeN}
\begin{lemma}\label{lem:B2PeInd}
$\mB_2\cup\mP_e$ is an independent set in $G$.
\end{lemma}
\journ{
\begin{proof} 
Let $x,y\in\mP_e$. If $x$ and $y$ are the end points of two different
paths in $\mP$, clearly, they are not adjacent. Since each path is
special, it has at least $3$ vertices which implies that if $x$ and $y$
are end points of the same path, then they are not adjacent. Hence,
$\mP_e$ induces an independent set in $G$. Let $x,y\in\mB_2$. By
definition, they have two neighbors each in $\mN$. Therefore, if $x$
and $y$ are adjacent, they induce a component in $\mR\cup\mB$, which
contradicts Statement $3$ of Lemma \ref{lem:componentsY}. Noting that
$\mP_e\subseteq\mS$ and $\mB_2\subseteq\mB$, from Observation
\ref{obs:finePart}.\ref{finePart:BDisjointS}, it follows that no vertex
in $\mB_2$ is adjacent to any vertex in $\mP_e$. Hence proved. \qed
\end{proof}
}
\begin{observation}\label{obs:B2PeN}
\journ{Consider a vertex $v\in\mB_2\cup\mP_e$. By the definition of $\mB_2$ and
Observation \ref{obs:RP2i}.\ref{RP2i:Pe2N}, it follows that $v$ is adjacent
to exactly two vertices in $\mN$. By Lemma \ref{lem:B2PeInd}, $v$ is not
adjacent to any vertex in $\mB_2\cup\mP_e$. Therefore, in the graph induced
by $\mB_2\cup\mP_e\cup\mN$, its degree is $2$. Lemma~\ref{lem:NCPi} implies
that every vertex in $\mN$ is adjacent to at most two vertices in
$\mB_2\cup\mP_e\cup\mN$. From these two observations, we can infer the
following about the graph induced by $\mB_2\cup\mP_e\cup\mN$.}
\conf{We have the following observations about the graph induced by
$\mB_2\cup\mP_e\cup\mN$.}
\begin{enumerate}
\item Its maximum degree is $2$ and thus is a collection of paths and
cycles.\label{B2PeN:B2PeNDelta2}
\item All the end points of paths (which also includes isolated vertices)
belong to $\mN$.\label{B2PeN:NEndIsolated}
\item A vertex in $\mN$ is adjacent to a vertex in $\mA$ only if it
is an end point of a path. \journ{The proof is as
follows: Let $v\in\mN$. It has at least one neighbor in $S$ (Observation
\ref{obs:primPart}.\ref{primPart:SN1}). If it has a neighbor in $\mA$,
then it can have at most one neighbor in $\mB_2\cup\mP_e\cup\mN$. Hence,
proved.}\label{B2PeN:NEndIsolatedA}
\end{enumerate}
\end{observation}

\begin{definition}{\bf\boldmath $\mN_e$ and $\mN_{int}$:}
\label{def:NeNint} $\mN$ is partitioned into $\mN_e$, the set of end
points of paths (which includes isolated vertices) and $\mN_{int}$, the
set of interior points of cycles and paths in $G[\mB_2\cup\mP_e\cup\mN]$.
\end{definition}
\journ{
In view of Observation \ref{obs:B2PeN}.\ref{B2PeN:NEndIsolatedA}, a vertex
in $\mN$ is adjacent to a vertex in $\mA$ only if it belongs to $\mN_e$.
}
\begin{definition}{\bf Type 1 and Type 2 cycles:}\label{def:type12}
Recall that by Observation \ref{obs:B2PeN}.\ref{B2PeN:B2PeNDelta2},
$\mB_2\cup\mP_e\cup\mN$ induces a collection of
cycles and paths. We classify the cycles in the following manner:
\begin{description}
\item[Type 1:] Cycles whose vertices alternate between $\mN$ and
$\mB_2\cup\mP_e$.
\item[Type 2:] Cycles which are not Type 1. From Lemma \ref{lem:B2PeInd},
it is easy to infer that such a cycle has at least one pair of adjacent
vertices which belong to $\mN$.
\end{description}
\end{definition}

\journ{
\begin{lemma}\label{lem:NAdjB2PeDisjointCP2e}
If a vertex $v\in\mN$ is adjacent to $2$ vertices in $\mB_2\cup\mP_e$,
then its remaining neighbor belongs to $\mP_{2i}$. In other words, $v$
has no neighbor in $\mC\cup\mP_{2e}$.
\end{lemma}
\begin{proof}
Let $v\in\mN$ be such that it is adjacent to two vertices
in $\mB_2\cup\mP_e$. Let $x_1$ and $x_2$ be these neighbors.
From Lemma \ref{lem:NCPi}, it follows that $v$ is adjacent to one
vertex in $\mC\cup\mP_i$. Let this vertex be $y$. We need to show that
$y\in\mP_{2i}$.

Suppose $T$ is the first special path or cycle chosen in Algorithm
\ref{alg:primPart} with $v$ as a neighbor. We will first show that $v$,
$x_1$ and $x_2$ are present in $V'$ when $T$ is being chosen. Since $T$
is the first component of $\mS$ with $v$ as its neighbor, it implies that
$v\in V'$ at that time. If any of $x_1$ and $x_2$ belongs to $\mB_2$, then,
it is in $V'$ because
$\mB_2\subseteq\mA_1\subseteq V'$ at any time in the algorithm. If any of
$x_1$ and $x_2$, say $x$ is in $\mP_e$, then again,
since $T$ is the first component of $\mS$ with $v$ as the neighbor,
it follows that $x$ is an end point of either $T$ or a special path
chosen after $T$. In either case, $x$ belongs to $V'$ when $T$ is
being chosen. The following observation is crucial for the proof:

\emph{Consider any $t\in\{y,x_1,x_2\}$. If $t\notin T$, it implies
that it is not adjacent to any vertex in $T$, since otherwise it would
be present in $\mN_1$ and this is not possible since by assumption
$y\in\mS$ and $x_1,x_2\in\mB_2\cup\mP_e$. }

We will now show that if $y\in\mC\cup\mP_{2e}$, it contradicts the
assumption that $T$ is a special cycle or path.

Let us suppose that $T$ is a special cycle. Since $x_1,x_2\notin\mC$,
it follows that $x_1,x_2\notin T$ and therefore, $y\in T$. Let $y'\in T$
be adjacent to $y$. Since $x_1\notin T$, it is not adjacent to any vertex
of $T$ and therefore, $(T\setminus\{y'\})\cup\{v,x_1\}$ induces a path with
$|T|+1$ vertices in $G[V']$ contradicting the fact that $T$ is a special
cycle. Therefore, from now on we will assume that $T$ is a special path. 

If $y\notin T$, it implies that at least one of $\{x_1,x_2\}$ belongs to
$T$. If only one of them, say $x_1\in T$, then, since $x_1$ has to be an
end point of $T$, $T\cup\{v\}$ induces a longer path in $G[V']$ and if
both $x_1,x_2\in T$, then, $T\cup\{v\}$ induces a cycle. In either case,
we have a contradiction to the fact that $T$ is a special path.

Suppose $y\in T$. Since we have assumed that $y\in\mC\cup\mP_{2e}$,
this means $y$ is a second end point of $T$. Let $p_e'$ be an end
point of $T$ adjacent to $y$ and let $p_e$ be the other end
point. If $p_e\in\{x_1,x_2\}$, then, $(T\setminus\{p_e'\})\cup\{v\}$
induces a cycle contradicting the fact that $T$ is a special path. If
$p_e\notin\{x_1,x_2\}$, it implies that at most one vertex from
$\{x_1,x_2\}$ can be an end point of $T$. Without loss of generality,
we assume that $x_1$ is not an end point of $T$. This implies that
$x_1\notin T$ and therefore, is not adjacent to any vertex in $T$.
Hence, $(T\setminus\{p_e'\})\cup\{v,x_1\}$ induces a path with $|T|+1$
vertices, again contradicting the fact that $T$ is a special path.
Therefore, $y\in\mP_{2i}$.  \qed
\end{proof}
}

\journ{
\begin{observation}\label{obs:B2PeNIsnMinus1}
We can assume that $|\mB_2\cup\mP_e\cup\mN|\le n-2$. This is because,
since $G$ is cubic, it has a cycle and therefore, we can always extract a
special cycle with at least $3$ vertices or a special path with at least
$4$ vertices from $V$. Thus, we can ensure that $|\mS\setminus\mP_e|\ge
2$. This implies that, $|\mB_2\cup\mP_e\cup\mN|\le n-2$.
\end{observation}
}

\subsection{A summary}
In this section, we partitioned the vertex set $V$ as follows:
\journ{$V=\mS\uplus\mN_1\uplus\mA_1$, where $\mS=\mC\uplus\mP$,
$\mN_1=\mR\uplus\mN$ and $\mA_1=\mB\uplus\mA$. Further,
$\mP=\mP_e\uplus\mP_{2e}\uplus\mP_{2i}$ and
$\mB=\mB_1\uplus\mB_2$. Therefore,} $V=\mC\uplus\mP_e\uplus\mP_{2e}
\uplus\mP_{2i}\uplus\mN \uplus\mR\uplus\mB_1\uplus\mB_2 \uplus\mA$. This
partitioning of $V$ is illustrated in Figure \ref{fig:cubicPartTree}.
Some of the observations and lemmas which we developed will be frequently
referred to in the sections to come. For the convenience of the reader,
we have tabulated them as follows. \journ{In Table \ref{tab:nonAdj}, we have
listed the pairs of sets $X,Y\subseteq V$ which satisfy the property that
in $G$ there is no edge between $X$ and $Y$. The relevant observation
or lemma is listed in the third column.} In Table \ref{tab:uniqNeighbor},
we list pairs of sets $X,Y\subseteq V$ such that in $G$, every vertex of
$X$ has at most one neighbor in $Y$. The corresponding observation or
lemma can be found in the third column. In the fourth column, we give
information on whether the vertex in $X$ has at most one neighbor or
exactly one neighbor in $Y$.
\begin{figure}[tb]
\begin{center}
\epsfig{file=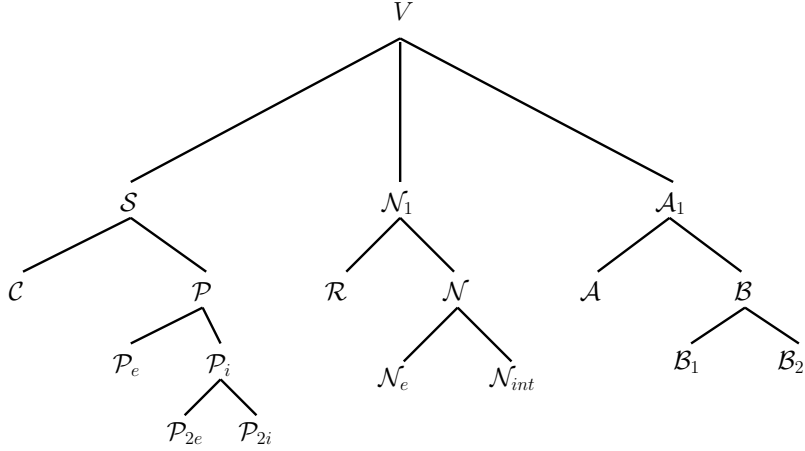,height=6cm}
\end{center}
\caption{The partition of the vertex set $V$ of the cubic graph.\journ{
The first and second level partitions are due to Algorithms
\ref{alg:primPart} and \ref{alg:finePart} respectively. The partitioning
of $\mP$ is covered in the beginning of Section \ref{sec:partS}. The
partitioning of $\mN$ and $\mB$ are due to Definitions \ref{def:NeNint}
and \ref{def:B1B2} respectively.}
\label{fig:cubicPartTree}}
\end{figure}
\journ{
\begin{table}[tb]
\begin{center}
\begin{tabular}{|r|c|c|l|}
\hline
1&$\mA$ & $\mC\cup\mP$ & using Observation \ref{obs:primPart}.\ref{primPart:ADisjointS} and the fact that $\mA\subseteq\mA_1$\\\hline
2&$\mA\cup\mN$ & $\mR$ & Observation \ref{obs:finePart}.\ref{finePart:RDisjointNA}(b)\\\hline
3&$\mA$ & $\mB$ & Observation \ref{obs:finePart}.\ref{finePart:BDisjointA}\\\hline
4&$\mC$ & $\mP$ & follows directly from the definitions of $\mC$ and $\mP$.\\\hline
5&$\mC\cup\mP_e\cup\mP_{2e}$ & $\mR$ & Lemma \ref{lem:RP2i}\\\hline
6&$\mC\cup\mP$ & $\mB$ & Observation \ref{obs:finePart}.\ref{finePart:BDisjointS}\\\hline
7&$\mP_e$ & $\mP_e$ & Lemma \ref{lem:B2PeInd}\\\hline
8&$\mP_e$ & $\mP_{2i}$ & follows directly from the definitions of $\mP_e$ and $\mP_{2i}$.\\\hline
9&$\mN_{int}$ & $\mA$ & Observation \ref{obs:B2PeN}.\ref{B2PeN:NEndIsolatedA} and Definition \ref{def:NeNint}\\\hline
10&$\mR$ & $\mR$ & Observation \ref{obs:finePart}.\ref{finePart:RDisjointNA}(a)\\\hline
\end{tabular}
\caption{{\bf Non-adjacency table:} In every row, there is no edge
between the set in the 1st column and the set in the 2nd column in
$G$.\label{tab:nonAdj}}
\end{center}
\end{table}
}
\begin{table}[tb]
\begin{center}
\begin{tabular}{|r|c|c|l|l|}
\hline
1&$\mN_e$ & $\mA$ & Observation \ref{obs:B2PeN}.\ref{B2PeN:NEndIsolatedA} and Definition \ref{def:NeNint}& at most one
neighbor\\\hline
2&$\mB_1$ & $\mN$ & Definition \ref{def:B1B2} & exactly one neighbor \\\hline
3&$\mC\cup\mP_{2e}$ & $\mN$ & Observation~\ref{obs:RP2i}.\ref{RP2i:CP2eDisjointR} & exactly one neighbor \\\hline
4&$\mP_{2i}$ & $\mR\cup\mN$ & Observation \ref{obs:RP2i}.\ref{RP2i:P2iRN} & exactly one neighbor\\\hline
5&$\mR$ & $\mP_{2i}$ & Lemma \ref{lem:RP2i} & exactly one neighbor\\\hline
\end{tabular}
\caption{{\bf Unique neighbor table:} In every row, each vertex
belonging to the set in the 1st column has either (a) at most one
neighbor OR (b) exactly one neighbor in the set given in the 2nd
column.\label{tab:uniqNeighbor}}
\end{center}
\end{table}

\section{Construction of a $3$-box representation of $G$}\label{sec:3box}
In order to give a $3$-box representation of $G$, we define three interval
graphs $I_1$, $I_2$ and $I_3$\journ{ and verify that $E(G)=E(I_1)\cap
E(I_2)\cap E(I_3)$}. \conf{The verification of $E(G)=E(I_1)\cap
E(I_2)\cap E(I_3)$ and the condition that two intersecting boxes touch
only at their boundaries is in the appendix.} Let $n=|V|$, the number
of vertices in $G$. For any $v\in V$, let $f(v,I_j)$, $j=1,2,3$ denote
the closed interval assigned to $v$ in the interval representation
of $I_j$. Further, let $l(v,I_j)$ and $r(v,I_j)$ denote the left and
right end points of $f(v,I_j)$ respectively. In each interval graph,
the interval assigned to a vertex is based on the set it belongs to in
the partition of $V$ (illustrated in Figure \ref{fig:cubicPartTree}).
\journ{We will also show that for every pair of adjacent vertices $x$
and $y$, in at least one interval graph $I_j$, $j\in\{1,2,3\}$, either
$l(x,I_j)=r(y,I_j)$ or $l(y,I_j)=r(x,I_j)$. This is sufficient to prove
that their corresponding boxes intersect only at their boundaries.}

\subsection{Construction of $I_1$}\label{sec:I1}
\subsubsection{Vertices of $\mA$} \label{sec:AI1}
We recall from Observation \ref{obs:primPart}.\ref{primPart:isolatedA}
that since $\mA\subseteq\mA_1$, it induces a collection of isolated
vertices and edges in $G$. 
\begin{definition}\label{def:PiA}
Let $\Pi_A$ be an ordering of $\mA$ which
satisfies the condition that the two end points of every (isolated)
edge are consecutively ordered. 
\end{definition}
\journ{
The intervals assigned to the vertices
of $\mA$ are as follows:
}
\paragraph{\bf An isolated vertex} $u$ is given a point interval as follows:
\begin{equation}\label{eqn:AIsolatedI1}
f(u,I_1)=[2n+\Pi_A(u),2n+\Pi_A(u)].
\end{equation}
\paragraph{\bf End points of an isolated edge} $(u,v)$: Without loss of
generality, let $\Pi_A(v)=\Pi_A(u)+1$. We assign the intervals to $u$ and
$v$ as follows:
\begin{eqnarray}
f(u,I_1)&=&[2n+\Pi_A(u),2n+\Pi_A(u)+0.5],\label{eqn:AEdgexI1} \\
f(v,I_1)&=&[2n+\Pi_A(v)-0.5,2n+\Pi_A(v)].\label{eqn:AEdgeyI1}
\end{eqnarray}
\journ{
\begin{observation}\label{obs:AATouch}
Let $x,y\in\mA$ such that $\Pi_A(x)<\Pi_A(y)$. If they are adjacent in
$G$, then, $r(x,I_1)=l(y,I_1)$. This follows from (\ref{eqn:AEdgexI1})
and (\ref{eqn:AEdgeyI1}).
\end{observation}

\begin{lemma}\label{lem:AGI1}
The graph induced by $\mA$ in $I_1$ and $G$ are identical, that is,
$I_1[\mA]=G[\mA]$.
\end{lemma}
\begin{proof}
From the interval assignments (\ref{eqn:AIsolatedI1}),
(\ref{eqn:AEdgexI1}) and (\ref{eqn:AEdgeyI1}), we observe the following:
for any $x\in\mA$, (a) the point $2n+\Pi_A(x)$ is an end point of $f(x,I_1)$
and (b) either $f(x,I_1)$ is a point interval or its length is $0.5$.

Suppose $x,y\in\mA$ such that $\Pi_A(x)<\Pi_A(y)$. From the above
observations it follows that $x$ and $y$ are adjacent in $I_1$ if and only
if $\Pi_A(y)=\Pi_A(x)+1$ and $f(x,I_1)=[2n+\Pi_A(x),2n+\Pi_A(x)+0.5]$
and $f(y,I_1)=[2n+\Pi_A(y)-0.5,2n+\Pi_A(y)]$. This implies that
$x$ was assigned an interval by (\ref{eqn:AEdgexI1}) and $y$, by
(\ref{eqn:AEdgeyI1}). This happens only if $x$ and $y$ are the end points
of an isolated edge in $G[\mA]$. Hence proved.  \qed
\end{proof}
}
\subsubsection{Vertices of ${\mB_2\cup\mP_e\cup\mN}$}\label{sec:B2PeNI1} 
We define an ordering on this set.
\begin{definition}\label{def:Pi1}
Let $\Pi_1$ be an ordering of $\mB_2\cup\mP_e\cup\mN$ which
satisfies the following properties. Let $S$ be a component induced by
$\mB_2\cup\mP_e\cup\mN$. We recall from Section \ref{sec:B2PeN} that $S$
is either a path or a cycle.
\begin{enumerate}
\item Let $S$ be a path with at least two vertices. Then, for one of
the natural orderings of the vertices of $S$, say $p_1p_2\ldots p_t$,
we have $\Pi_1(p_i)=\Pi_1(p_{i-1})+1$, $2\le i\le t$.
\item Suppose $S$ is a cycle. We recall that $S$ is either a Type 1 or a
Type 2 cycle (Definition \ref{def:type12}). Consider one of the natural
orderings of the vertices of $S$, say $c_1c_2\ldots c_tc_1$ such that
if $S$ is a Type 1 cycle, then $c_1\in\mN$ and if it is a Type 2 cycle,
then $c_1,c_t\in\mN$. Then, we have, $\Pi_1(c_i)=\Pi_1(c_{i-1})+1$,
$2\le i\le t$.
\item Let
$\mB_2\cup\mP_e\cup\mN= \Lambda_{\textrm{Type
1}}\uplus\Lambda_{\textrm{Type 2}}\uplus\Lambda_{\textrm{Paths}}$ where,
$\Lambda_{\textrm{Type 1}}$, $\Lambda_{\textrm{Type 2}}$ and $\Lambda_{\textrm{Paths}}$ are the sets of vertices
belonging to Type 1 cycles, Type 2 cycles and paths respectively. Then, we
have $\Pi_1({\Lambda_{\textrm{Type 1}}})<\Pi_1({\Lambda_{\textrm{Type
2}}})<\Pi_1({\Lambda_{\textrm{Paths}}})$.
\end{enumerate}
\end{definition}
\journ{
It is easy to verify that such an ordering exists. We also
infer that if $S_1$ and $S_2$ are two different components of
$G[\mB_2\cup\mP_e\cup\mN]$, then, either $\Pi_1(S_1)<\Pi_1(S_2)$ or
$\Pi_1(S_2)<\Pi_1(S_1)$.
\begin{observation}\label{obs:Pi1nMinus1}
$\forall z\in\mB_2\cup\mP_e\cup\mN$, $\Pi_1(z)<n-1$. This follows from the
fact that $|\mB_2\cup\mP_e\cup\mN|<n-1$ (Observation
\ref{obs:B2PeNIsnMinus1}).
\end{observation}
}
The interval assignments for the vertices of $\mB_2\cup\mP_e\cup\mN$
are as follows:
\paragraph{\bf For a vertex in a Type 1 cycle:} Let $S=c_1c_2\ldots
c_tc_1$ be a Type 1 cycle such that $\Pi_1(c_{i+1})=\Pi_1(c_i)+1$,
$1\le i<t$.
\begin{eqnarray}
&&f(c_1,I_1)=\left[\Pi_1(c_1),\Pi_1(c_t)\right];\label{eqn:B2PeNType1c1I1}\\
&&f(c_i,I_1)=\left[\Pi_1(c_i),\Pi_1(c_i)+1\right], 1<i<t;\label{eqn:B2PeNType1ciI1}\\
&&f(c_t,I_1)=\left[\Pi_1(c_t),\Pi_1(c_t)+0.5\right].\label{eqn:B2PeNType1ctI1}
\end{eqnarray}
\journ{
\begin{observation}\label{obs:Type1I1}
Suppose $S=c_1c_2\ldots c_tc_1$ is a Type 1 cycle such that
$\Pi_1(c_{i+1})=\Pi_1(c_i)+1$, $1\le i<t$.
\begin{enumerate}
\item For $1<i<t$, $r(c_i,I_1)=l(c_{i+1},I_1)$ and therefore,
$I_1[S\setminus c_1]=G[S\setminus c_1]$. This is because, from
(\ref{eqn:B2PeNType1ciI1}) and (\ref{eqn:B2PeNType1ctI1}),
$r(c_i,I_1)=\Pi_1(c_i)+1=\Pi_1(c_{i+1})=l(c_{i+1},I_1)$.\label{Type1I1:SMinusc1Touch}
\item $I_1[S]$ is a supergraph of $G[S]$. The proof is as follows:
$c_1$ is adjacent to all the other vertices of $S$. This follows
from (\ref{eqn:B2PeNType1c1I1}): $l(c_i,I_1)=\Pi_1(c_i)\in f(c_1,I_1)$.
From Point 1, $c_i$ is adjacent to $c_{i+1}$, $2\le i<t$. Hence, proved.
\label{Type1I1:Type1SuperI1}
\item Let $x\in S_x$ and $y\in S_y$, where $S_x$ and $S_y$ induce
different Type 1 cycles. Then, $(x,y)\notin E(I_1)$. The proof is as
follows: Without loss of generality, let $\Pi_1(S_x)<\Pi_1(S_y)$. From
(\ref{eqn:B2PeNType1c1I1}--\ref{eqn:B2PeNType1ctI1}), it follows that
$\ds r(x,I_1)\le \max_{a\in S_x}\Pi_1(a)+0.5<\min_{b\in S_y}\Pi_1(b)\le
l(y,I_1)$. Therefore, $(x,y)\notin E(I_1)$. \label{Type1I1:nonAdjType1}
\end{enumerate}
\end{observation}
}
\paragraph{\bf For a vertex in a Type 2 cycle:} Let $S=c_1c_2\ldots
c_tc_1$ be a Type 2 cycle such that $\Pi_1(c_{i+1})=\Pi_1(c_i)+1$, $1\le
i<t$. We recall from the definition of $\Pi_1$ that $c_1,c_t\in\mN$.
They are assigned intervals as follows:
\begin{eqnarray}
f(c_1,I_1)&=&\left[n+\Pi_1(c_1),n+\Pi_1(c_t)\right], \label{eqn:B2PeNType2c1I1} \\
f(c_t,I_1)&=&\left[n+\Pi_1(c_t),n+\Pi_1(c_t)+0.5\right]. \label{eqn:B2PeNType2ctI1}
\end{eqnarray}
The remaining vertices are assigned intervals as follows. For $1<i<t$,
\begin{eqnarray}
&&\textrm{if $c_i\in\mN$, then, }f(c_i,I_1)= \left[n+\Pi_1(c_i),n+\Pi_1(c_i)+1\right],\label{eqn:B2PeNType2othersNI1}\\
&&\textrm{if $c_i\in\mB_2\cup\mP_e$, then, }f(c_i,I_1)= \left[n,n+\Pi_1(c_i)+1\right].\label{eqn:B2PeNType2othersB2PeI1}
\end{eqnarray}
\journ{
\begin{observation}\label{obs:Type2I1}
Let $S=c_1c_2\ldots c_tc_1$ be a Type 2 cycle such
that $\Pi_1(c_{i+1})=\Pi_1(c_i)+1$, $1\le i<t$. 
\begin{enumerate}
\item $I_1[S]$ is a supergraph of $G[S]$. The proof is as follows: From
(\ref{eqn:B2PeNType2c1I1}--\ref{eqn:B2PeNType2othersB2PeI1}), $\forall
c\in S$, $n+\Pi_1(c)\in f(c,I_1)$. From (\ref{eqn:B2PeNType2c1I1}),
we note that $\forall c\in S$, $n+\Pi_1(c)\in f(c_1,I_1)$ and
therefore, $c_1$ is adjacent to all the other vertices of $S$. From
(\ref{eqn:B2PeNType2othersNI1}) and (\ref{eqn:B2PeNType2othersB2PeI1}),
for $1<i<t$, $r(c_{i},I_1)=n+\Pi_1(c_{i})+1=n+\Pi_1(c_{i+1})\in
f(c_{i+1},I_1)$. Therefore, $c_i$ is adjacent to $c_{i+1}$,
$1<i<t$.\label{Type2I1:Type2SuperI1}
\item Let $x,y\in S$ be two adjacent vertices in $G$ such that
neither $x$ nor $y$ is $c_1$. If $\Pi_1(x)>\Pi_1(y)$ and
$x\in\mN$, then, $l(x,I_1)=r(y,I_1)$.  This follows by noting
that $\Pi_1(x)=\Pi_1(y)+1$ and subsequently applying it in
(\ref{eqn:B2PeNType2ctI1}--\ref{eqn:B2PeNType2othersB2PeI1}).\label{Type2I1:touch}
\end{enumerate}
\end{observation}
}
\paragraph{\bf For a vertex in a path:} Let $S=p_1p_2\ldots p_t$ be a
path such that $\Pi_1(c_{i+1})=\Pi_1(c_i)+1$, $1\le i<t$. By Definition
\ref{def:NeNint}, $p_1,p_t\in\mN_e$. From Table \ref{tab:uniqNeighbor}
(row 1), they can be adjacent to at most one vertex in $\mA$. Taking
this into consideration, the interval assignments are as follows: Let
$p\in\{p_1,p_t\}\subseteq\mN_e$:
\begin{eqnarray}
&&\textrm{if $p$ is not adjacent to any vertex in $\mA$, then, } f(p,I_1)=\left[n+\Pi_1(p),2n\right], \label{eqn:B2PeNp1ptnoneI1}\\
&&\textrm{if $p$ is adjacent to a vertex $a$ in $\mA$, then, } f(p,I_1)=\left[n+\Pi_1(p),l(a,I_1)\right].\label{eqn:B2PeNp1ptxI1}
\end{eqnarray}
\journ{
Note that $f(a,I_1)$ is already defined in
(\ref{eqn:AIsolatedI1}--\ref{eqn:AEdgeyI1}). Moreover, $l(a,I_1)>
2n\ge n+\Pi_1(p)$. Therefore, $f(p,I_1)$ is well-defined in
(\ref{eqn:B2PeNp1ptxI1}).} If $p$ is an interior point of the path,
its interval assignment is as follows:
\begin{eqnarray}
&&\textrm{if $p\in\mN_{int}$, then, }f(p,I_1)=\left[n+\Pi_1(p),n+\Pi_1(p)+1\right],\label{eqn:B2PeNpInteriorNI1}\\
&&\textrm{if $p\in\mB_2\cup\mP_e$, then, }f(p,I_1)=\left[n,n+\Pi_1(p)+1\right].\label{eqn:B2PeNpInteriorB2PeI1} 
\end{eqnarray}
\journ{
\begin{observation}\label{obs:PathI1}
Let $S=p_1p_2\ldots p_t$ be a path such that
$\Pi_1(c_{i+1})=\Pi_1(c_i)+1$, $1\le i<t$. 
\begin{enumerate}
\item $I_1[S]$ is a supergraph of $G[S]$. The proof is as follows:
For all $p\in S$, $n+\Pi_1(p),n+\Pi_1(p)+1\in f(p,I_1)$. This is easy to infer from
(\ref{eqn:B2PeNp1ptnoneI1}--\ref{eqn:B2PeNpInteriorB2PeI1}) and the fact
that $\Pi_1(p)<n-1$ (Observation \ref{obs:Pi1nMinus1}). This implies
that for $1\le i<t$, $p_i$ is adjacent to $p_{i+1}$. Hence,
proved.\label{PathI1:PathSuperI1}
\item Let $x,y\in S$ be two adjacent vertices in $G$. If
$\Pi_1(x)>\Pi_1(y)$, $x\in\mN$ and $y\in\mB_2\cup\mP_e$,
then, $l(x,I_1)=r(y,I_1)$. This follows by noting that
$\Pi_1(x)=\Pi_1(y)+1$ and subsequently applying it in
(\ref{eqn:B2PeNp1ptnoneI1}--\ref{eqn:B2PeNpInteriorB2PeI1}).\label{PathI1:touch}
\end{enumerate}
\end{observation}

\begin{observation}\label{obs:BunchI1}
We have some observations regarding the intervals assigned to vertices
of $\mB_2\cup\mP_e\cup\mN$. We repeatedly make use of Observation
\ref{obs:Pi1nMinus1}.
\begin{enumerate}
\item If $z$ belongs to a Type 1 cycle, then, (a) $l(z,I_1)=\Pi_1(z)$
and (b) $1\le l(z,I_1)<r(z,I_1)<n$. The proof is as follows:
Let $z$ belong to the Type 1 cycle $S_z=c_1c_2\ldots c_tc_1$
such that $\Pi_1(c_{i+1})=\Pi_1(c_i)+1$, $1\le i<t$. From
(\ref{eqn:B2PeNType1c1I1}--\ref{eqn:B2PeNType1ctI1}), it immediately
follows that $l(z,I_1)=\Pi_1(z)$, $l(z,I_1)<r(z,I_1)$ and
$r(z,I_1)\le\Pi_1(c_t)+0.5<(n-1)+0.5<n$. \label{BunchI1:Type1}

\item If $z\in\mN$ belongs to a Type 2 cycle or a path then,
$l(z,I_1)=n+\Pi_1(z)$ and therefore, $n<l(z,I_1)<2n-1$. This follows from
(\ref{eqn:B2PeNType2c1I1}--\ref{eqn:B2PeNType2othersNI1}) for a Type
2 cycle and (\ref{eqn:B2PeNp1ptnoneI1}--\ref{eqn:B2PeNpInteriorNI1})
for a path.\label{BunchI1:NType2Path}

\item If $z\in\mB_2\cup\mP_e$ belongs to a Type 2 cycle or a path
then, $l(z,I_1)=n$ and $r(z,I_1)=n+\Pi_1(z)+1<2n$. This follows
from (\ref{eqn:B2PeNType2othersB2PeI1}) for a Type 2 cycle and
(\ref{eqn:B2PeNpInteriorB2PeI1}) for a path.\label{BunchI1:B2PeType2Path}.

\item If $x,y\in\mN$ such that $\Pi_1(x)<\Pi_1(y)$, then,
$l(x,I_1)+1\le l(y,I_1)$. The proof is as follows: If $x$ and
$y$ belong to Type 1 cycles, then by Point 1 in this observation,
$l(x,I_1)+1=\Pi_1(x)+1\le\Pi_1(y)=l(y,I_1)$. If $y$ belongs to a
Type 1 cycle, then by Definition \ref{def:Pi1} (Point 3), $x$ also belongs
to a Type 1 cycle. Therefore, it is not possible that $y$ is in a Type 1
cycle and $x$ is not. If $x$ is in a Type 1 cycle and $y$ is not, then,
$l(x,I_1)+1=\Pi_1(x)+1<n+\Pi_1(y)=l(y,I_1)$. Finally, if both
belong to a Type 2 cycle or a path, then, $l(x,I_1)+1=n+\Pi_1(x)+1\le
n+\Pi_1(y)=l(y,I_1)$ (from Point \ref{BunchI1:NType2Path} in this
observation).\label{BunchI1:xLessy}

\item If $x$ belongs to a Type 1 cycle and $y$ belongs
to either a Type 2 cycle or a path, then, $x$ and $y$ are
not adjacent in $I_1$. The proof is as follows: From Point
\ref{BunchI1:Type1} in this observation, $r(x,I_1)<n$ and from Points
\ref{BunchI1:NType2Path} and \ref{BunchI1:B2PeType2Path}, $l(y,I_1)\ge
n$.\label{BunchI1:Type1SepType2PathI1}

\item If $x\in\mN$ is adjacent to $a\in\mA$, then,
$r(x,I_1)=l(a,I_1)$. The proof is as follows: By Table \ref{tab:nonAdj}
(row 9), $x\in\mN_e$ and by Table \ref{tab:uniqNeighbor} (row 1),
$a$ should be the only neighbor of $x$ in $\mA$. The interval
assignment for $x$ is given in (\ref{eqn:B2PeNp1ptxI1}) where,
$r(x,I_1)=l(a,I_1)$.\label{BunchI1:ANTouch}
\end{enumerate}
\end{observation}

\begin{lemma}\label{lem:AB2PeNSuperI1}
$I_1[\mA\cup\mB_2\cup\mP_e\cup\mN]$ is a supergraph of
$G[\mA\cup\mB_2\cup\mP_e\cup\mN]$.
\end{lemma}
\begin{proof}
Let $x,y\in\mA\cup\mB_2\cup\mP_e\cup\mN$ be two adjacent vertices
in $G$. If $x,y\in\mA$, then, by Lemma \ref{lem:AGI1} they are
adjacent in $I_1$. Let $x,y\in\mB_2\cup\mP_e\cup\mN$. They have
to belong to the same component in $G[\mB_2\cup\mP_e\cup\mN]$,
which is either a Type 1 cycle, Type 2 cycle or a path. By
Observations \ref{obs:Type1I1}.\ref{Type1I1:Type1SuperI1},
\ref{obs:Type2I1}.\ref{Type2I1:Type2SuperI1} and
\ref{obs:PathI1}.\ref{PathI1:PathSuperI1},
$x$ and $y$ are adjacent in $I_1$.

Now it remains to be shown that if $x\in\mA$ and
$y\in\mB_2\cup\mP_e\cup\mN$, then they are adjacent in $I_1$. Noting that
$\mB_2\subseteq\mB$ and $\mP_e\subseteq\mP$, $y\notin\mB_2\cup\mP_e$
(see Table \ref{tab:nonAdj}, rows 1 and 3). Therefore,
$y\in\mN$. By Observation \ref{obs:BunchI1}.\ref{BunchI1:ANTouch},
$r(y,I_1)=l(x,I_1)$. Therefore, $x$ is adjacent to $y$ in $I_1$. Hence
proved. \qed
\end{proof}
}
\subsubsection{Vertices of $\mB_1\cup\mP_{2e}$} 
Let $v\in\mB_1\cup\mP_{2e}$. From Table \ref{tab:uniqNeighbor} (rows 2 and
3), we note that $v$ has a unique neighbor in $\mN$, say $v'$. $f(v',I_1)$
is already defined in Section \ref{sec:B2PeNI1}.
\begin{eqnarray} 
&&\textrm{if $v\in\mB_1$, then, } f(v,I_1)=\left[0,l(v',I_1)\right]\label{eqn:B1I1},\\
&&\textrm{if $v\in\mP_{2e}$, then, } f(v,I_1)=\left[-1,l(v',I_1)\right]\label{eqn:P2eI1}.
\end{eqnarray}
\journ{
\begin{lemma}\label{lem:B1P2eI1}
For any $x\in\mB_1\cup\mP_{2e}$, $r(x,I_1)>n$ and therefore, $[0,n]\subset
f(x,I_1)$.
\end{lemma}
\begin{proof}
From Table \ref{tab:uniqNeighbor} (rows 2 and 3), $x$ has a unique
neighbor in $\mN$, say $x'$. From (\ref{eqn:B1I1}) and (\ref{eqn:P2eI1}),
$r(x,I_1)=l(x',I_1)$. We will now show that $x'$ does not belong to a
Type~1 cycle in the graph induced by $\mB_2\cup\mP_e\cup\mN$.  Suppose
$x\in\mB_1$. Since $\mN\subseteq\mN_1$, by Lemma~\ref{lem:NCPi}, $x'$ is
adjacent to at least one vertex in $\mC\cup\mP_i$. Since $x\in\mB_1$, $x'$
cannot be adjacent to two vertices in $\mB_2\cup\mP_e$ and hence cannot
belong to a Type~1 cycle. Now suppose $x\in\mP_{2e}$. If $x'$ belongs
to a Type 1 cycle, then it has two neighbors in $\mB_2\cup\mP_e$. By
Lemma~\ref{lem:NAdjB2PeDisjointCP2e}, the remaining neighbor of $x'$,
that is $x$, does not belong to $\mP_{2e}$, which is a contradiction.

Thus, $x'$ belongs to either a Type~2 cycle or
a path in $\mB_2\cup\mP_e\cup\mN$. From Observation
\ref{obs:BunchI1}.\ref{BunchI1:NType2Path}, $l(x',I_1)>n$
and therefore, $r(x,I_1)>n$. From the interval assignments for $x$ in
(\ref{eqn:B1I1}) and (\ref{eqn:P2eI1}), it immediately follows that
$[0,n]\subset f(x,I_1)$.\qed
\end{proof}
}
\subsubsection{Vertices of ${\mR}$}\label{sec:RI1} 
\conf{$ \forall v\in\mR, f(v,I_1)=\left[-1,n\right]$.}
\journ{\begin{equation}\label{eqn:RI1}
\forall v\in\mR, f(v,I_1)=\left[-1,n\right].
\end{equation}}
\journ{
Consider the set of vertices which have been assigned intervals
until now: $\mA\cup(\mB_2\cup\mP_e\cup\mN)\cup(\mB_1\cup\mP_{2e})\cup\mR=
V\setminus(\mP_{2i}\cup\mC)$.
\begin{lemma}\label{lem:VminusCP2iSuperI1}
$I_1[V\setminus(\mP_{2i}\cup\mC)]$
is a supergraph of $G[V\setminus(\mP_{2i}\cup\mC)]$.
\end{lemma}
\begin{proof}
Let $x,y\in V\setminus(\mP_{2i}\cup\mC)$ be two adjacent vertices
in $G$. If $x,y\in\mA\cup\mB_2\cup\mP_e\cup\mN$, then by Lemma
\ref{lem:AB2PeNSuperI1}, $x$ is adjacent to $y$ in $I_1$. Let
$x\in\mB_1\cup\mP_{2e}\cup\mR$. By Lemma \ref{lem:B1P2eI1}, $[0,n]\subset
f(x,I_1)$.

If $y\in\mB_1\cup\mP_{2e}$, again by Lemma \ref{lem:B1P2eI1},
$[0,n]\subset f(y,I_1)$ and if $y\in\mR$, then by (\ref{eqn:RI1}),
$[0,n]\subset f(y,I_1)$ and therefore, $x$ is adjacent to $y$. If
$y\in\mA$, then, since $\mB_1\subseteq\mB$ and $\mP_{2e}\subseteq\mP$,
by Table \ref{tab:nonAdj} (rows 1--3) $x$ cannot be adjacent to $y$
in $G$. Suppose $y\in\mB_2\cup\mP_e$. If $y$ belongs to a Type 1
cycle, then by Observation \ref{obs:BunchI1}.\ref{BunchI1:Type1}(b),
$f(y,I_1)\subset[1,n]$. Otherwise, from Observation
\ref{obs:BunchI1}.\ref{BunchI1:B2PeType2Path}, $l(y,I_1)=n$. In either
case, $x$ is adjacent to $y$ in $I_1$.  Finally, suppose $y\in\mN$. By
Table \ref{tab:nonAdj} (row 2), $x\notin\mR$, which implies that
$x\in\mB_1\cup\mP_{2e}$. From Table \ref{tab:uniqNeighbor} (rows 2 and
3), $y$ is the unique neighbor of $x$ in $\mN$. By (\ref{eqn:B1I1})
and (\ref{eqn:P2eI1}), $r(x,I_1)=l(y,I_1)$ and therefore, $x$ and $y$
are adjacent in $I_1$. Hence proved. \qed
\end{proof}
}
\subsubsection{Vertices of ${\mP_{2i}}$} 
Suppose $v\in\mP_{2i}$. Let $v'$ be its unique neighbor in $\mR\cup\mN$
(see Table \ref{tab:uniqNeighbor} row 4). Note that $f(v',I_1)$
is already defined Sections \ref{sec:B2PeNI1} and \ref{sec:RI1}.
\begin{eqnarray}
&&\textrm{if $v'\in\mR$, then, } f(v,I_1)=\left[-1,-1\right], \label{eqn:P2iRI1}\\
&&\textrm{if $v'\in\mN$, then, } f(v,I_1)=\left[-1,l(v',I_1)\right].\label{eqn:P2iNI1}
\end{eqnarray}
\journ{
\begin{lemma}\label{lem:VminusCSuperI1}
$I_1[V\setminus\mC]$ is a supergraph of $G[V\setminus\mC]$.
\end{lemma}
\begin{proof}
Let $x,y\in V\setminus\mC$ be two adjacent vertices in $G$. By Lemma
\ref{lem:VminusCP2iSuperI1}, if $x,y\in V\setminus(\mC\cup\mP_{2i})$,
then, they are adjacent in $I_1$.  Suppose $x\in\mP_{2i}$. By
definition, $x$ is an interior vertex of a special path in $\mP$ and
therefore, it is adjacent to two vertices from the path. Moreover,
it is not adjacent to any end point of this path since otherwise it
would be present in $\mP_{2e}$. Therefore, two of the neighbors of
$x$ are in $\mP_{2i}\cup\mP_{2e}$. By Table \ref{tab:uniqNeighbor}
(row 4), the third neighbor of $x$ has to be in $\mR\cup\mN$. From
this we infer that $y\in\mP_{2i}\cup\mP_{2e}\cup\mR\cup\mN$. Note
that $l(x,I_1)=-1$ by (\ref{eqn:P2iRI1}) and (\ref{eqn:P2iNI1}).
If $y\in\mP_{2i}\cup\mP_{2e}\cup\mR$, then by the interval assignments
(\ref{eqn:P2eI1}--\ref{eqn:P2iNI1}), it follows that $l(y,I_1)=-1$. If
$y\in\mN$, then, by (\ref{eqn:P2iNI1}), $r(x,I_1)=l(y,I_1)$. In either
case, $x$ is adjacent to $y$ in $I_1$. Hence proved. \qed
\end{proof}

\begin{observation}\label{obs:P2iRTouch}
If $x\in\mP_{2i}$ and $y\in\mR$ are adjacent in $G$, then,
$r(x,I_1)=l(y,I_1)$. From Table \ref{tab:uniqNeighbor} (row 4), $y$ is the
only neighbor of $x$ in $\mR\cup\mN$. The interval assignment for $x$ is
given by (\ref{eqn:P2iRI1}) and for $y$ by (\ref{eqn:RI1}), from which it
follows that $r(x,I_1)=l(y,I_1)=-1$.
\end{observation}
}
\subsubsection{Vertices of ${\mC}$} 
We recall that $\mC$ induces a collection of cycles in $G$.
\begin{definition}{\bf\boldmath Notation $\eta(\cdot)$ and
special vertex:}\label{def:specialVertex} We recall from Table
\ref{tab:uniqNeighbor} (row 3) that every vertex $x\in\mC$
has a unique neighbor in $\mN$. We denote this neighbor by
$\eta(x)$. We define a vertex $c\in C$ as the special vertex of $C$ if
$l(\eta(c),I_1)=\ds\min_{c'\in C}l(\eta(c'),I_1)$. Note that $\eta(c)$
is already assigned an interval in Section \ref{sec:B2PeNI1}.
\end{definition}
Suppose $C$ is a cycle in $\mC$. Let $C=c_1c_2\ldots c_tc_1$ be a natural
ordering of the vertices of $C$ such
that $c_1$ is the special vertex of $C$. The interval assignments are
as follows:
\begin{equation}\label{eqn:CI1}
\begin{array}{ll}
f(c_1,I_1)=\left[l(\eta(c_1),I_1),l(\eta(c_1),I_1)\right],\\
f(c_i,I_1)=\left\{
   \begin{array}{ll}
      \left[l(\eta(c_1),I_1),l(\eta(c_i),I_1)+0.5\right], & i=2,t,\\
      \left[l(\eta(c_1),I_1)+0.5,l(\eta(c_i),I_1)+0.5\right], & \textrm{otherwise}.\\
   \end{array}\right.
\end{array}
\end{equation}
\journ{
Since $l(\eta(c_1),I_1)<l(\eta(c_i),I_1)$ for $i\ne1$, we observe that
the intervals are well-defined.

\begin{observation}\label{obs:CI1}
Let $C=c_1c_2\ldots c_tc_1$ be a cycle in $\mC$ with $c_1$ being the
special vertex.
\begin{enumerate}
\item $c_1$ is adjacent to only $c_2$ and $c_t$ in $I_1$. Further,
$r(c_1,I_1)=l(c_2,I_1)=l(c_t,I_1)$. This is easy to
infer from (\ref{eqn:CI1}).\label{CI1:special}
\item $C\setminus\{c_1\}$ is a clique in $I_1$. Since $l(\eta(c_1),I_1)\le
l(\eta(c),I_1)$, $c\in C$, from (\ref{eqn:CI1}), it follows that
$\forall c\in C\setminus\{c_1\}$, $l(\eta(c_1)+0.5,I_1)\in
f(c,I_1)$.\label{CI1:nonSpecial}
\item\label{CI1:cEtacI1}
For every vertex $c\in\mC$, $l(\eta(c),I_1)\in f(c,I_1)$ and therefore,
$c$ is adjacent to $\eta(c)$ in $I_1$. The proof follows: Since $c_1$ is
a special vertex, by definition, $l(\eta(c_1),I_1)\le l(\eta(c),I_1)$,
and therefore from the interval assignments in (\ref{eqn:CI1}),
$l(\eta(c),I_1)\in f(c,I_1)$.
\end{enumerate}
\end{observation}

\noindent Now we will show the following:
\begin{lemma} \label{lem:superI1}
$I_1$ is a supergraph of $G$.
\end{lemma}
\begin{proof}
Let $x,y\in V$ be two adjacent vertices in $G$. If $x,y\in V\setminus\mC$,
then, by Lemma \ref{lem:VminusCSuperI1}, $x$ and $y$ are adjacent
in $I_1$. If $x,y\in\mC$, then, they belong to the same component
(which is a special cycle) in $G[\mC]$. From Observation \ref{obs:CI1}
(Points \ref{CI1:special} and \ref{CI1:nonSpecial}), we can infer that
they are adjacent in $I_1$. The only case remaining is the one in which
one vertex is in $\mC$ and the other in $V\setminus\mC$.

Suppose $x\in\mC$ and $y\in V\setminus\mC$. Since $x$ belongs to a special
cycle in $\mC$, it has two neighbors in the cycle. The remaining neighbor
is $y$. By Table \ref{tab:uniqNeighbor} (row 3), $y\in\mN$. Moreover,
by the notation introduced in Definition \ref{def:specialVertex},
$y=\eta(x)$. From Observation \ref{obs:CI1}.\ref{CI1:cEtacI1}, it follows
that $y$ is adjacent to $x$ in $I_1$. Hence proved.\qed
\end{proof}

Recall that we need to show that $E(G)=E(I_1)\cap E(I_2)\cap E(I_3)$. For
this, we will have to prove that every missing edge in $G$ is missing
in at least one of the three interval graphs. Note that there is no
edge between $\mA$ and $V\setminus(\mA\cup\mN)$ (Table \ref{tab:nonAdj}
rows 1--3). Now we will show that all these missing edges of $G$ are
also missing in $I_1$.  \begin{lemma}\label{lem:AANnonAdj} Let $x\in\mA$
and $y\in V\setminus(\mA\cup\mN)$. Then, $(x,y)\notin E(I_1)$.
\end{lemma}
\begin{proof}
From (\ref{eqn:AIsolatedI1}--\ref{eqn:AEdgeyI1}), $l(x,I_1)>2n$. Now
we will show that $r(y,I_1)\le2n$, from which it immediately
follows that $(x,y)\notin E(I_1)$. If $y\in\mB_2\cup\mP_e$,
from Observations \ref{obs:BunchI1}.\ref{BunchI1:Type1} and
\ref{obs:BunchI1}.\ref{BunchI1:B2PeType2Path}, $r(y,I_1)\le2n-1$. If
$y\in\mB_1\cup\mP_{2e}$, then by (\ref{eqn:B1I1}) and
(\ref{eqn:P2eI1}), $r(y,I_1)=l(z,I_1)$ for some $z\in\mN$. If
$y\in\mP_{2i}$, then by (\ref{eqn:P2iRI1}) and (\ref{eqn:P2iNI1}),
either $r(y,I_1)=-1$ or $r(y,I_1)=l(z,I_1)$ for some $z\in\mN$. If
$y\in\mC$, then from (\ref{eqn:CI1}), $r(y,I_1)\le l(z,I_1)+0.5$
for some $z\in\mN$.  If $y\in\mR$, then by (\ref{eqn:RI1}),
$r(y,I_1)=n$. From Observations \ref{obs:BunchI1}.\ref{BunchI1:Type1}
and \ref{obs:BunchI1}.\ref{BunchI1:NType2Path}, $\forall z\in\mN$,
$l(z,I_1)<2n-1$ and therefore, it follows that in each case $r(y,I_1)<2n$.
Thus, $(x,y)\notin E(I_1)$. Hence proved. \qed
\end{proof}
}
\subsection{Construction of $I_2$}\label{sec:I2}
\subsubsection{Vertices of ${\mA}$} 
We recall the interval assignment for $\mA$ in $I_1$ (see Section
\ref{sec:AI1}). Let $\overline{\Pi_A}$ be the reverse of $\Pi_A$. The
interval assignments for vertices of $\mA$ in $I_2$ are as follows:
\paragraph{\bf An isolated vertex} $u$ is given a point interval as follows:
\begin{equation}\label{eqn:AIsolatedI2}
f(u,I_2)=[n+\overline{\Pi_A}(u),n+\overline{\Pi_A}(u)].
\end{equation}
\paragraph{\bf End points of an isolated edge} $(u,v)$: Without
loss of generality, let $\Pi_A(v)=\Pi_A(u)+1$. This implies that
$\overline{\Pi_A}(v)=\overline{\Pi_A}(u)-1$. We assign the intervals to
$u$ and $v$ as follows:
\begin{equation}\label{eqn:AEdgeI2}
\begin{array}{ll}
f(u,I_2)=[n+\overline{\Pi_A}(u)-0.5,n+\overline{\Pi_A}(u)],\\
f(v,I_2)=[n+\overline{\Pi_A}(v),n+\overline{\Pi_A}(v)+0.5].
\end{array}
\end{equation}
\journ{
Note that the two intervals intersect at
$n+\overline{\Pi_A}(u)-0.5=n+\overline{\Pi_A}(v)+0.5$.

\begin{observation}\label{obs:AGI2}
The graph induced by $\mA$ in $I_2$ and $G$ are identical, that is,
$I_2[\mA]=G[\mA]$. The proof is similar to that of Lemma \ref{lem:AGI1}.
\end{observation}
}
\subsubsection{Vertices of ${\mN}$} 
Let $v\in\mN$. From Table \ref{tab:uniqNeighbor} (row 1), $v$ is adjacent
to at most one vertex in $\mA$.
\begin{eqnarray}
&&\textrm{if $v$ is not adjacent to any vertex in $\mA$, then, }
f(v,I_2)=[0,n], \label{eqn:NnoneI2} \\
&&\textrm{if $v$ is adjacent to vertex $a$ in $\mA$, then, }
f(v,I_2)=\left[0,l(a,I_2)\right].\label{eqn:NAI2}
\end{eqnarray}
\journ{
Note that $l(a,I_2)$ is already defined in (\ref{eqn:AIsolatedI2})
and (\ref{eqn:AEdgeI2}) and satisfies, $l(a,I_2)>n$. Hence, we have the
following observation.
\begin{observation}\label{obs:NCliqueI2}
In $I_2$, $\forall x\in\mN$, $[0,n]\subseteq f(x,I_2)$.
\end{observation}

\begin{lemma}\label{lem:ANSuperI2}
$I_2[\mA\cup\mN]$ is a supergraph of $G[\mA\cup\mN]$.
\end{lemma}
\begin{proof}
Let $x,y\in\mA\cup\mN$ be two adjacent vertices in $G$. If $x,y\in\mA$,
from Observation~\ref{obs:AGI2}, $(x,y)\in E(I_2)$. If $x,y\in\mN$,
from Observation \ref{obs:NCliqueI2}, $(x,y)\in E(I_2)$. If $x\in\mN$
and $y\in\mA$, then by Table \ref{tab:nonAdj} (row 9) and Table
\ref{tab:uniqNeighbor} (row 1), $x\in\mN_e$ and $y$ is its only neighbor
in $\mA$. From (\ref{eqn:NAI2}), $r(x,I_2)=l(y,I_2)$ and therefore,
$(x,y)\in E(I_2)$.\qed
\end{proof}

\begin{lemma}\label{lem:ANnonAdj}
If $x\in\mN$ and $y\in\mA$ such that $(x,y)\notin E(G)$, then,
$(x,y)\notin E(I_1\cap I_2)$.
\end{lemma}
\begin{proof}
Suppose $x$ is not adjacent to any vertex in $\mA$. In $I_2$, by
(\ref{eqn:NnoneI2}), $r(x,I_2)=n$ and by (\ref{eqn:AIsolatedI2}) and
(\ref{eqn:AEdgeI2}), $l(y,I_2)>n$ and therefore, $(x,y)\notin I_2$.
Let us assume that $x$ is adjacent to some vertex in $\mA$, say
$a$. From Table \ref{tab:nonAdj} (row 9), $x\in\mN_e$ and by Table
\ref{tab:uniqNeighbor} (row 1), $a$ is the only neighbor of $x$
in $\mA$. From the interval assignment in (\ref{eqn:B2PeNp1ptxI1}),
$r(x,I_1)=l(a,I_1)$ and from (\ref{eqn:NAI2}) $r(x,I_2)=l(a,I_2)$. If
$\Pi_A(a)<\Pi_A(y)$, then, $l(a,I_1)<l(y,I_1)$ (this is easy to
infer from (\ref{eqn:AIsolatedI1}--\ref{eqn:AEdgeyI1})) and therefore,
$(x,y)\notin I_1$. Otherwise, since $\Pi_A(a)>\Pi_A(y)$, it implies that
$\overline{\Pi_A}(a)<\overline{\Pi_A}(y)$ which in turn implies that
$l(a,I_2)<l(y,I_2)$ (see interval assignments in (\ref{eqn:AIsolatedI2})
and (\ref{eqn:AEdgeI2})). Therefore, $(x,y)\notin E(I_2)$. \qed
\end{proof}
}
\subsubsection{Vertices of ${\mC\cup\mP}$} 
We recall that $\mC\cup\mP$ induces a collection of special cycles and
special paths in $G$.

\begin{definition}\label{def:Pi2}
$\Pi_2$ is an ordering of $\mC\cup\mP$ such that the following properties
are satisfied:
\begin{enumerate}
\item Let $P$ be a path in $\mP$. For a natural ordering of $P$, say
$p_1p_2\ldots p_t$, we have $\Pi_2(p_{i})= \Pi_2(p_{i-1})+1$, $2\le
i\le t$.
\item Suppose $C$ is a cycle in $\mC$. For a natural ordering of $C$, say
$c_1c_2\ldots c_tc_1$, where $c_1$ is the special vertex (recall Definition
\ref{def:specialVertex}), we have $\Pi_2(c_{i})= \Pi_2(c_{i-1})+1$,
$2\le i\le t$.
\end{enumerate}
\end{definition}
\journ{
It is easy to see that such an ordering $\Pi_2$ exists. Also note that if
$S_1$ and $S_2$ are two different components of $G[\mC\cup\mP]$, then,
either $\Pi_2(S_1)<\Pi_2(S_2)$ or $\Pi_2(S_2)<\Pi_2(S_1)$. The interval
assignments are as follows:
}
\paragraph{\bf For the vertices of a path:} Suppose $P=p_1p_2\ldots p_t$
such that $\Pi_2(p_{i+1})= \Pi_2(p_i)+1$, $1\le i<t$.
\begin{equation}\label{eqn:PI2}
\begin{array}{l}
f(p_i,I_2)=\left[\Pi_2(p_i),\Pi_2(p_i)+1\right], 1\le i<t,\\
f(p_t,I_2)=\left[\Pi_2(p_t),\Pi_2(p_t)+0.5\right].
\end{array}
\end{equation}
\journ{
\begin{observation}\label{obs:PSuperI2}
Let $P=p_1p_2\ldots p_t$ be a special path from $\mP$ such that
$\Pi_2(p_{i+1})= \Pi_2(p_{i})+1$, $1\le i<t$. Then, for $1<i\le t$,
$l(p_i,I_2)=r(p_{i-1},I_2)$ and hence $I_2[P]=G[P]$.
\end{observation}
}
\paragraph{\bf For the vertices of a cycle:} Suppose $C=c_1c_2\ldots
c_tc_1$ such that $\Pi_2(c_{i+1})= \Pi_2(c_{i})+1$, $1\le i<t$.
\begin{equation}
\begin{array}{l}\label{eqn:CI2}
f(c_1,I_2)= \left[\Pi_2(c_1),\Pi_2(c_t)\right],\\
f(c_i,I_2)= \left[\Pi_2(c_i),\Pi_2(c_i)+1\right],1<i<t,\\
f(c_t,I_2)= \left[\Pi_2(c_t),\Pi_2(c_t)+0.5\right].
\end{array}
\end{equation}
\journ{
\begin{observation}\label{obs:CI2}
Suppose $C=c_1c_2\ldots c_tc_1$ is a special cycle from $\mC$ such that
$\Pi_2(c_{i+1})=\Pi_2(c_i)+1$, $1\le i<t$. 
\begin{enumerate}
\item For $2\le i\le t$, $l(c_{i+1},I_2)=r(c_i,I_2)$ and therefore,
$I_2[C\setminus c_1]=G[C\setminus c_1]$.\label{CI2:touch}
\item $I_2[C]$ is a supergraph of $G[C]$. The proof is as follows: $c_1$
is adjacent to all the other vertices of $C$ in $I_2$ and from Point 1
in this observation, $I_2[C\setminus c_1]=G[C\setminus
c_1]$.\label{CI2:CSuperI2}
\end{enumerate}
\end{observation}

\begin{observation}\label{obs:CP0nI2}
For every $x\in\mC\cup\mP$ and $y\in\mN$, $f(x,I_2)\subset f(y,I_2)$. The
proof is as follows: $G$ is a cubic graph whereas $G[\mC\cup\mP]$
has maximum degree $2$ and therefore, $\mC\cup\mP\ne V$. This implies that
$\forall z\in\mC\cup\mP$, $\Pi_2(z)<n$. Taking this into consideration,
from interval assignments (\ref{eqn:PI2}) and (\ref{eqn:CI2}) we can infer
that $0<l(x,I_2)<r(x,I_2)<n$ and therefore, $f(x,I_2)\subset[0,n]$. From
Observation \ref{obs:NCliqueI2}, $[0,n]\subseteq f(y,I_2)$. Hence proved.
\end{observation}

\begin{lemma}\label{lem:CPnonAdjI2}
If $x,y\in\mC\cup\mP$ belong to different components in $G[\mC\cup\mP]$,
then, $(x,y)\notin E(I_2)$.
\end{lemma}
\begin{proof}
Let $x\in S_x$ and $y\in S_y$, where $S_x$ and $S_y$ are two different
components of $G[\mC\cup\mP]$. Without loss of generality we will assume
that $\Pi_2(S_x)<\Pi_2(S_y)$.  Irrespective of whether $S_x$ (or $S_y$)
induces a path or a cycle in $G[\mC\cup\mP]$, from (\ref{eqn:PI2})
and (\ref{eqn:CI2}), it follows that $\ds r(x,I_2)\le\max_{a\in
S_x}\Pi_2(a)+0.5<\min_{b\in S_y}\Pi_2(b)\le l(y,I_2)$. Therefore,
$(x,y)\notin E(I_2)$.  \qed
\end{proof}

\begin{lemma}\label{lem:CPI1I2}
The graph induced by $\mC\cup\mP$ in $G$ and $I_1\cap I_2$ are identical,
that is, $G[\mC\cup\mP]=(I_1\cap I_2)[\mC\cup\mP]$.
\end{lemma}
\begin{proof}
Let $x,y\in\mC\cup\mP$. First we will show that if $(x,y)\in E(G)$,
then $(x,y)\in E(I_1\cap I_2)$. Clearly, from Lemma \ref{lem:superI1},
$(x,y)\in E(I_1)$. Since $x$ and $y$ are adjacent in $G$, they belong to
the same path or cycle in $G[\mC\cup\mP]$, say $S$. From Observations
\ref{obs:PSuperI2} and \ref{obs:CI2}.\ref{CI2:CSuperI2}, it follows that
$I_2[S]$ is a supergraph of $G[S]$. Therefore, $(x,y)\in E(I_2)$. Hence,
$(x,y)\in E(I_1\cap I_2)$.

Now we will show that if $(x,y)\notin E(G)$, then, either $(x,y)\notin
E(I_1)$ or $(x,y)\notin E(I_2)$. There are two cases to consider:
(1) $x$ and $y$ belong to different components in $G[\mC\cup\mP]$
and (2) they belong to the same component. If it is Case (1), then by
Lemma \ref{lem:CPnonAdjI2}, $(x,y)\notin E(I_2)$. If it is Case (2),
then, let $x,y\in S$, where $S$ is a component of $G[\mC\cup\mP]$. If
$S$ is a special path, then by Observation \ref{obs:PSuperI2},
$I_2[S]=G[S]$ and therefore, $(x,y)\notin E(I_2)$.  Suppose $S$ is a
special cycle. Let $S=c_1c_2\ldots c_tc_1$, where $c_1$ is the special
vertex. If neither $x$ nor $y$ is $c_1$, then, they are not adjacent in
$I_2$. This is because, by Observation \ref{obs:CI2}.\ref{CI2:touch},
$I_2[C\setminus\{c_1\}]=G[C\setminus\{c_1\}]$. Suppose $x=c_1$, then
clearly, $y\ne c_2,c_t$. By Observation \ref{obs:CI1}.\ref{CI1:special},
in $I_1$, $c_1$ is adjacent to only $c_2$ and $c_t$. Thus, $x$ and $y$
are not adjacent in $I_1$. Hence, $(x,y)\notin E(I_1\cap I_2)$. \qed
\end{proof}

\begin{lemma}\label{lem:ANCPSuperI2}
$I_2[\mA\cup\mN\cup\mC\cup\mP]$ is a super graph of
$G[\mA\cup\mN\cup\mC\cup\mP]$.
\end{lemma}
\begin{proof}
Let $x,y\in\mA\cup\mN\cup\mC\cup\mP$ be two adjacent vertices in $G$. If
$x,y\in\mA\cup\mN$, then by Lemma \ref{lem:ANSuperI2}, $(x,y)\in E(I_2)$
and if $x,y\in\mC\cup\mP$, then, from Lemma \ref{lem:CPI1I2} we can infer
that $(x,y)\in E(I_2)$. Therefore, we will assume that $x\in\mA\cup\mN$
and $y\in\mC\cup\mP$. By Table \ref{tab:nonAdj} (row 1), $x\notin\mA$
and hence, $x\in\mN$. From Observation \ref{obs:CP0nI2}, $f(y,I_2)\subset
f(x,I_2)$. Therefore, $(x,y)\in E(I_2)$.  \qed
\end{proof}
}
\subsubsection{Vertices of ${\mR\cup\mB}$} \label{sec:RBI2}
From Lemma \ref{lem:componentsY}, we recall that each component in
$\mR\cup\mB$ is isomorphic to one of the graphs shown in Figure
\ref{fig:componentsY}. Further, each component contains exactly
one vertex from $\mR$ and is uniquely identified by it; by the notation
introduced in Lemma \ref{lem:componentsY}, for every
$u\in\mR$, $\Gamma(u)$ denotes the component containing $u$ in
$G[\mR\cup\mB]$.
\begin{definition}{\bf\boldmath Notation $\beta(\cdot)$:}\label{def:baseVertex}
In the graph induced by $\mR\cup\mB$, consider each component $\Gamma(u)$,
$u\in\mR$. From Table \ref{tab:uniqNeighbor} (row 5), $u$ is adjacent
to a unique vertex in $\mP_{2i}$. We denote this vertex by $\beta(u)$.
\end{definition}

\paragraph{\bf Interval assignments for vertices of $\mR\cup\mB$:} Let us
consider a component of $G[\mR\cup\mB]$, say $\Gamma(u)$, $u\in\mR$. From
(\ref{eqn:PI2}), we note that $\beta(u)$ is assigned a unit interval
in $I_2$. The interval assignments for the vertices of $\Gamma(u)$
is illustrated in Figure \ref{fig:componentsYInterval}.

\journ{
\begin{remark}\label{rem:RBI2}
Let $u\in\mR$. Every vertex of $\Gamma(u)$ is assigned a strict
sub-interval of $f(\beta(u),I_2)$ and none of these intervals contains
any end point of $f(\beta(u),I_2)$.
\end{remark}
}

\begin{figure}[tb]
\begin{center}
\epsfig{file=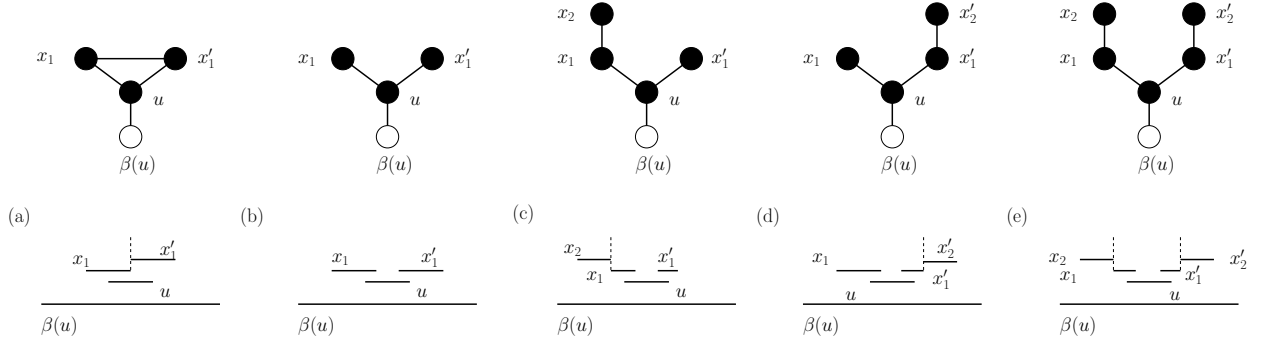,width=16.5cm}
\end{center}
\caption{The interval assignments for each component $\Gamma(u)$,
$u\in\mR$ induced by $\mR\cup\mB$ in the interval graph $I_2$. \journ{The
dotted vertical lines are used to indicate that the concerned
intervals intersect exactly at their end points, that is, in (a)
$r(x_1,I_2)=l(x_1',I_2)$, in (c) $r(x_2,I_2)=l(x_1,I_2)$, in (d)
$l(x_2',I_2)=r(x_1',I_2)$ and in (e) $r(x_2,I_2)=l(x_1,I_2)$ and
$l(x_2',I_2)=r(x_1',I_2)$}. \label{fig:componentsYInterval}}
\end{figure}
\journ{
\begin{lemma}\label{lem:RBI2}
Let $z\in\mR$ and $\Gamma(z)$ be a component of $G[\mR\cup\mB]$. 
\begin{enumerate}
\item Every vertex in $\Gamma(z)$ is adjacent to only one vertex in
$\mC\cup\mP$ and that is $\beta(z)$.
\item The graph induced by $\mR\cup\mB$ in $I_2$ and in $G$ are identical,
that is, $I_2[\mR\cup\mB]=G[\mR\cup\mB]$.
\end{enumerate}
\end{lemma}
\begin{proof}
Consider the vertex $\beta(z)$. Let $f_o(\beta(z),I_2)$ denote the open
interval $\left(l(\beta(z),I_2),r(\beta(z),I_2)\right)$. From Definition
\ref{def:baseVertex}, we recall that $\beta(z)\in\mP_{2i}$. We first
prove the following:
\begin{clm}\label{clm:ppDashI2}
Let $p\in\mC\cup\mP$ such that $p\ne\beta(z)$. Then, $f_o(\beta(z),I_2)\cap
f(p,I_2)=\varnothing$.
\end{clm}
\begin{proof}
If $\beta(z)$ and $p$ belong to different components in $G[\mC\cup\mP]$,
then by Lemma \ref{lem:CPnonAdjI2}, $(\beta(z),p)\notin E(I_2)$ and
therefore, their intervals do not intersect. Suppose $\beta(z)$ and
$p$ are in the same component. Since $\beta(z)\in\mP_{2i}$, $\beta(z)$
and $p$ belong to a special path. By (\ref{eqn:PI2}), it follows that
$f(\beta(z),I_2)$ and $f(p,I_2)$ can intersect only at $l(\beta(z),I_2)$
or $r(\beta(z),I_2)$. Hence proved. \bqed
\end{proof}
By Remark \ref{rem:RBI2}, for every $x\in\Gamma(z)$, $f(x,I_2)\subseteq
f_o(\beta(z),I_2)$. This implies that $x$ is adjacent to $\beta(z)$
in $I_2$ and by Claim \ref{clm:ppDashI2}, $x$ is not adjacent to any
other vertex from $\mC\cup\mP$. Thus, we have proved the first statement.

Suppose $x\in\Gamma(z)$ and $y\in\Gamma(z')$, where $z\ne z'$. We
first note that $\beta(z)\ne\beta(z')$. This is because, since
$\beta(z)\in\mP_{2i}$, in $G[\mC\cup\mP]$, it is the interior vertex
of a special path and therefore, two of its neighbors belong to
$\mP$. Since its remaining neighbor is $z$, it cannot be adjacent to
$z'$. By Remark \ref{rem:RBI2}, $f(x,I_2)\subseteq f_o(\beta(z),I_2)$ and
$f(y,I_2)\subseteq f_o(\beta(z'),I_2)$. From Claim \ref{clm:ppDashI2},
$f_o(\beta(z),I_2)\cap f_o(\beta(z'),I_2)=\varnothing$ and therefore,
$x$ is not adjacent to $y$. From Figure \ref{fig:componentsYInterval},
it is easy to see that $I_2[\Gamma(z)]=G[\Gamma(z)]$ $\forall
z\in\mR$. Therefore, $I_2[\mR\cup\mB]=G[\mR\cup\mB]$.  \bqed
\end{proof}

\begin{observation}\label{obs:RBCP}
Here are some immediate consequences of Lemma \ref{lem:RBI2}.
\begin{enumerate}
\item If $x\in\mR$ and $y\in\mP_{2i}$ such that $(x,y)\notin E(G)$, then,
$(x,y)\notin E(I_2)$. This follows from noting that $y\ne\beta(x)$
and subsequently applying Lemma \ref{lem:RBI2} (Statement
1).\label{RBCP:RP2iI2}
\item For any $x\in\mR\cup\mB$ and $y\in\mC\cup\mP_e\cup\mP_{2e}$,
$(x,y)\notin E(I_2)$. The proof is as follows: Let $x\in\Gamma(z)$,
$z\in\mR$. From Lemma \ref{lem:RBI2} (Statement 1), $x$ is
not adjacent to any vertex in $\mC\cup\mP$ other than $\beta(z)$ in
$I_2$. From Definition \ref{def:baseVertex}, $\beta(z)\in\mP_{2i}$ and
therefore, $y\ne\beta(z)$.\label{RBCP:RBCPeP2eI2}
\end{enumerate}
\end{observation}

\begin{lemma} \label{lem:superI2}
$I_2$ is a supergraph of $G$.
\end{lemma}
\begin{proof}
Let $x$ and $y$ be two adjacent vertices in $G$. If $x,y\in
\mA\cup\mN\cup\mC\cup\mP$, then, by Lemma \ref{lem:ANCPSuperI2},
$(x,y)\in E(I_2)$. If $x,y\in\mR\cup\mB$, then, by Lemma \ref{lem:RBI2}
(Statement 2), $(x,y)\in E(I_2)$. The only remaining case is when
$x\in\mA\cup\mN\cup\mC\cup\mP$ and $y\in\mR\cup\mB$. Let $\Gamma(z)$,
$z\in\mR$ be the component in $G[\mR\cup\mB]$ containing $y$.  By Table
\ref{tab:nonAdj} (rows 2, 3, 5 and 6), $x\notin\mA\cup\mC$ and therefore,
$x\in\mP\cup\mN$.

Suppose $x\in\mP$. By Table \ref{tab:nonAdj} (row 6), $y\notin\mB$. Since
by assumption $y\in\mR\cup\mB$, it implies that $y\in\mR$, which in turn
implies that $z=y$ as $z$ is the only vertex from $\mR$ in $\Gamma(z)$.
From Table \ref{tab:nonAdj} (row 5), we can infer that $x\in\mP_{2i}$. Now
by Definition \ref{def:baseVertex}, $x=\beta(y)$. By Lemma \ref{lem:RBI2}
(Statement 1), $y$ is adjacent to $x$ in $I_2$.  Finally, if $x\in\mN$,
then by Observation \ref{obs:CP0nI2}, $f(x,I_2)\supset f(\beta(z),I_2)$,
since $\beta(z)\in\mP_{2i}\subset\mP$. By Remark \ref{rem:RBI2},
$f(y,I_2)\subset f(\beta(z),I_2)$. Therefore, $f(y,I_2)\subset f(x,I_2)$
and $x$ is adjacent to $y$ in $I_2$. \qed
\end{proof}
}
\subsection{Construction of $I_3$}\label{sec:I3}
\subsubsection{Vertices of ${\mB_2\cup\mP_e\cup\mN}$}\label{sec:B2PeNI3}
We recall the notations and interval assignments developed for this set
in $I_1$, in particular the definition of $\Pi_1$ (Definition
\ref{def:Pi1}).

\paragraph{\bf For a vertex in a Type 1 cycle:} Let $S=c_1c_2\ldots
c_tc_1$ be a Type 1 cycle such that $\Pi_1(c_{i+1})=\Pi_1(c_i)+1$,
$1\le i<t$. We recall that $c_1\in\mN$ and $c_t\in\mB_2\cup\mP_e$. The
interval assignments are as follows: For $1\le i<t$,
\begin{eqnarray}
&&\textrm{if $c_i\in\mN$, then, }f(c_i,I_3)=\left[\Pi_1(c_i),\Pi_1(c_i)+1\right],\label{eqn:B2PeNType1ciNI3}\\
&&\textrm{if $c_i\in\mB_2$, then, }f(c_i,I_3)=\left[\Pi_1(c_i),n\right],\label{eqn:B2PeNType1ciB2I3}\\
&&\textrm{if $c_i\in\mP_e$, then, }f(c_i,I_3)=\left[\Pi_1(c_i),n+1\right].\label{eqn:B2PeNType1ciPeI3}
\end{eqnarray}
The interval assignment for $c_t$ is as follows:
\begin{eqnarray}
&&\textrm{if $c_t\in\mB_2$, then, }f(c_t,I_3)=\left[\Pi_1(c_1)+1,n\right],\label{eqn:B2PeNType1ctB2I3}\\
&&\textrm{if $c_t\in\mP_e$, then, }f(c_t,I_3)=\left[\Pi_1(c_1)+1,n+1\right],\label{eqn:B2PeNType1ctPeI3}
\end{eqnarray}
\journ{
\begin{observation}\label{obs:Type1I3}
Consider a Type 1 cycle $S=c_1c_2\ldots c_tc_1$, such that
$\Pi_1(c_{i+1})=\Pi_1(c_i)+1$, $1\le i<t$.  
\begin{enumerate}
\item $I_3[S]$ is a supergraph of $G[S]$. The proof is as follows:
From Observation \ref{obs:Pi1nMinus1}, $\Pi_1(c_i)<n-1$. Using this in
(\ref{eqn:B2PeNType1ciNI3}--\ref{eqn:B2PeNType1ciPeI3}), we note that for
$1\le i<t$, $l(c_i,I_3)=\Pi_1(c_i)$ and $\Pi_1(c_i)+1\in f(c_i,I_3)$,
and therefore, $c_i$ is adjacent to $c_{i+1}$ in $I_3$. Next, from
(\ref{eqn:B2PeNType1ctB2I3}) and (\ref{eqn:B2PeNType1ctPeI3}), it is
easy to infer that $f(c_t,I_3)$ contains $\Pi_1(c_i)+1$, $1\le i\le
t$. Therefore, $c_t$ is adjacent to all the other vertices of the
cycle. Hence proved.\label{Type1I3:Type1SuperI3}
\item $r(c_1,I_3)=l(c_2,I_3)=l(c_t,I_3)$.\label{Type1I3:touch}
\end{enumerate}
\end{observation}

\begin{lemma}\label{lem:Type1NI1I3}
Let $x\in\mN$ and $y\in\mB_2\cup\mP_e\cup\mN$ be such that $x\in S_x$
and $y\in S_y$ where both $S_x$ and $S_y$ induce Type 1 cycles. If
$(x,y)\notin E(G)$, then $(x,y)\notin E(I_1\cap I_3)$.
\end{lemma}
\begin{proof}
If $S_x\ne S_y$, by Observation
\ref{obs:Type1I1}.\ref{Type1I1:nonAdjType1}, $(x,y)\notin E(I_1)$.
Suppose $S_x=S_y=S$. Let $S=c_1c_2\ldots c_tc_1$ such that
$\Pi_1(c_{i+1})=\Pi_1(c_{i})+1$, $1\le i<t$. If neither $x$ nor $y$ is
$c_1$, then by Observation \ref{obs:Type1I1}.\ref{Type1I1:SMinusc1Touch},
they are not adjacent in $I_1$ since $I_1[S\setminus
\{c_1\}]=G[S\setminus\{c_1\}]$. If $x=c_1$, then, $y=c_j$
for some $j\ne 2,t$. In $I_3$, noting that $c_1\in\mN$, from
(\ref{eqn:B2PeNType1ciNI3}), $r(x,I_3)=r(c_1,I_3)=\Pi_1(c_1)+1$.
From (\ref{eqn:B2PeNType1ciNI3}--\ref{eqn:B2PeNType1ciPeI3})
$l(y,I_3)=\Pi_1(y)=\Pi_1(c_j)>\Pi_1(c_1)+1$. Hence, $(x,y)\notin
E(I_3)$.\qed
\end{proof}
}
\paragraph{\bf For a vertex in a Type 2 cycle:} Let $S=c_1c_2\ldots
c_tc_1$ be a Type 2 cycle such that $\Pi_1(c_{i+1})=\Pi_1(c_i)+1$,
$1\le i<t$. We recall from Definition \ref{def:Pi1} (Point 2) that
$c_1,c_t\in\mN$.  They are assigned intervals as follows:
\begin{eqnarray}
&&f(c_1,I_3)=\left[\Pi_1(c_1),\Pi_1(c_1)+1\right],\label{eqn:B2PeNType2c1I3}\\
&&f(c_t,I_3)=\left[\Pi_1(c_1)+1,\Pi_1(c_t)+0.5\right],\label{eqn:B2PeNType2ctI3}
\end{eqnarray}
For $c_i$, $1<i<t$,
\begin{eqnarray}
&&\textrm{if $c_i\in\mN$, then, } f(c_i,I_3)=\left[\Pi_1(c_i),\Pi_1(c_i)+1\right],\label{eqn:B2PeNType2NI3}\\
&&\textrm{if $c_i\in\mB_2$, then, } f(c_i,I_3)=\left[\Pi_1(c_i),n\right],\label{eqn:B2PeNType2B2I3}\\ 
&&\textrm{if $c_i\in\mP_e$, then, } f(c_i,I_3)=\left[\Pi_1(c_i),n+1\right].\label{eqn:B2PeNType2PeI3} 
\end{eqnarray}
\journ{
\begin{observation}\label{obs:Type2I3}
Let $S=c_1c_2\ldots c_tc_1$ be a Type 2 cycle such that
$\Pi_1(c_{i+1})=\Pi_1(c_i)+1$, $1\le i<t$.
\begin{enumerate}
\item $I_3[S]$ is a supergraph of $G[S]$. The proof is as follows:
Recalling from Observation \ref{obs:Pi1nMinus1} that $\Pi_1(z)<n-1$,
$\forall z\in\mB_2\cup\mP_e\cup\mN$, from (\ref{eqn:B2PeNType2c1I3})
and (\ref{eqn:B2PeNType2NI3}--\ref{eqn:B2PeNType2PeI3}),
for $1\le i<t$, $\Pi_1(c_{i}),\Pi_1(c_{i})+1\in
f(c_{i},I_3)$. Therefore, $c_i$ is adjacent to $c_{i+1}$, $1\le
i<t$. From (\ref{eqn:B2PeNType2ctI3}), for $1\le i<t$, $\Pi_1(c_i)+1\in
f(c_t,I_3)$. Therefore, $c_t$ is adjacent to all the other vertices
of $S$. Hence proved.\label{Type2I3:Type2SuperI3}
\item $r(c_1,I_3)=l(c_2,I_3)=l(c_t,I_3)$.\label{Type2I3:c1Touch}
\item Let $x,y\in S$ be two adjacent vertices in $G$ such that
neither $x$ nor $y$ is $c_t$. If $\Pi_1(x)<\Pi_1(y)$ and
$x\in\mN$, then, $r(x,I_3)=l(y,I_3)$. This follows by noting
that $\Pi_1(y)=\Pi_1(x)+1$ and subsequently applying it in
(\ref{eqn:B2PeNType2c1I3}) and (\ref{eqn:B2PeNType2NI3}--\ref{eqn:B2PeNType2PeI3}).\label{Type2I3:touch}
\end{enumerate}
\end{observation}
}
\paragraph{\bf For a vertex in a path:} Let $S=p_1p_2\ldots p_t$ be
a path such that $\Pi_1(p_{i+1})=\Pi_1(p_i)+1$, $1\le i<t$. We recall
from Observation \ref{obs:B2PeN}.\ref{B2PeN:NEndIsolated} and Definition
\ref{def:NeNint} that $p_1,p_t\in\mN_e$. The interval assignment for $p_t$
is as follows:
\begin{eqnarray}
f(p_t,I_3)=\left[\Pi_1(p_t),\Pi_1(p_t)+0.5\right].\label{eqn:B2PeNPathptI3}
\end{eqnarray}
The interval assignment for $p_i$, $i<t$ is as follows:
\begin{eqnarray}
&&\textrm{if $p_i\in\mN$, then,
}f(p_i,I_3)=\left[\Pi_1(p_i),\Pi_1(p_i)+1\right], \label{eqn:B2PeNPathpiNI3}\\
&&\textrm{if $p_i\in\mB_2$, then, }f(p_i,I_3)=\left[\Pi_1(p_i),n\right],\label{eqn:B2PeNPathpiB2I3}\\ 
&&\textrm{if $p_i\in\mP_e$, then, }f(p_i,I_3)=\left[\Pi_1(p_i),n+1\right].\label{eqn:B2PeNPathpiPeI3}
\end{eqnarray}
\journ{
\begin{observation}\label{obs:PathI3}
Let $S=p_1p_2\ldots p_t$ be a path such that
$\Pi_1(p_{i+1})=\Pi_1(p_i)+1$, $1\le i<t$.
\begin{enumerate}
\item $I_3[S]$ is a supergraph of $G[S]$.The proof is as follows:
Since by Observation \ref{obs:Pi1nMinus1}, $\Pi_1(p_i)<n-1$, we can infer from the interval assignments
in (\ref{eqn:B2PeNPathpiNI3}--\ref{eqn:B2PeNPathpiPeI3}) that,
for $i<t$, the points $\Pi_1(p_i)$ and $\Pi_1(p_i)+1$ belong to
$f(p_i,I_3)$. From (\ref{eqn:B2PeNPathptI3}), the point $\Pi_1(p_t)$
belongs to $f(p_t,I_3)$. This implies that for $1\le i<t$, $p_i$ is
adjacent to $p_{i+1}$. Hence, proved.\label{PathI3:PathSuperI3}
\item Let $x,y\in S$ be two adjacent vertices in $G$. If
$\Pi_1(x)<\Pi_1(y)$, $x\in\mN$ and $y\in\mB_2\cup\mP_e\cup\mN$,
then, $r(x,I_3)=l(y,I_3)$. This follows by noting that
$\Pi_1(y)=\Pi_1(x)+1$ and subsequently applying it in
(\ref{eqn:B2PeNPathptI3}--\ref{eqn:B2PeNPathpiPeI3}).\label{PathI3:touch}
\end{enumerate}
\end{observation}

\begin{observation}\label{obs:BunchI3}
Here are some observations regarding the intervals assigned to vertices
of $\mB_2\cup\mP_e\cup\mN$. We recall from Observation \ref{obs:Pi1nMinus1}
that for any $z\in\mB_2\cup\mP_e\cup\mN$, $\Pi_1(z)<n-1$.
\begin{enumerate}
\item If $z\in\mN$, $\Pi_1(z)+0.5\le r(z,I_3)\le \Pi_1(z)+1<n$. This
follows from the interval assignments in (\ref{eqn:B2PeNType1ciNI3}),
(\ref{eqn:B2PeNType2c1I3}--\ref{eqn:B2PeNType2NI3}) and
(\ref{eqn:B2PeNPathptI3}--\ref{eqn:B2PeNPathpiNI3}).
\label{BunchI3:NLess2n}
\item If $z\in\mB_2\cup\mP_e$ belongs to a Type 2 cycle or a path, then,
$l(z,I_3)=\Pi_1(z)<n-1$. This follows from
(\ref{eqn:B2PeNType2B2I3}--\ref{eqn:B2PeNType2PeI3}) for Type 2 cycles and
(\ref{eqn:B2PeNPathpiB2I3}--\ref{eqn:B2PeNPathpiPeI3}) for paths
respectively.\label{BunchI3:B2PeLess2n}
\item If $z\in\mB_2$, from (\ref{eqn:B2PeNType1ciB2I3}),
(\ref{eqn:B2PeNType1ctB2I3}), (\ref{eqn:B2PeNType2B2I3}) and
(\ref{eqn:B2PeNPathpiB2I3}), $r(x,I_3)=n$.\label{BunchI3:B2I3}
\item If $z\in\mP_e$, from (\ref{eqn:B2PeNType1ciPeI3}),
(\ref{eqn:B2PeNType1ctPeI3}), (\ref{eqn:B2PeNType2PeI3}) and
(\ref{eqn:B2PeNPathpiPeI3}), $r(x,I_3)=n+1$.\label{BunchI3:PeI3}
\end{enumerate}
\end{observation}

\begin{observation}\label{obs:B2PeNSuperI3}
$I_3[\mB_2\cup\mP_e\cup\mN]$ is a supergraph of
$G[\mB_2\cup\mP_e\cup\mN]$. The proof is as follows:
Let $S$ be a component of $G[\mB_2\cup\mP_e\cup\mN]$. It
is either a Type 1 cycle, Type 2 cycle or path. From
Observations \ref{obs:Type1I3}.\ref{Type1I3:Type1SuperI3},
\ref{obs:Type2I3}.\ref{Type2I3:Type2SuperI3} and
\ref{obs:PathI3}.\ref{PathI3:PathSuperI3},
it follows that $I_3[S]$ is a supergraph of $G[S]$.
\end{observation}

\begin{lemma}\label{lem:Type2PathNI1I3}
Let $x\in\mN$ and $y\in\mB_2\cup\mP_e\cup\mN$ be such that $x\in S_x$
and $y\in S_y$ where $S_x$ and $S_y$ induce a Type 2 cycle or a path. If
$(x,y)\notin E(G)$, then $(x,y)\notin E(I_1\cap I_3)$.
\end{lemma}
\begin{proof}
We will consider the following two cases separately: (1) $y\in\mN$ and
(2) $y\in\mB_2\cup\mP_e$.
\paragraph{\boldmath $y\in\mN$:} If $S_x\ne S_y$, then,
$(x,y)\notin E(I_3)$. The proof is as follows: Without loss
of generality let $\Pi_1(S_x)<\Pi_1(S_y)$. From the interval
assignments for a vertex of $\mN$ in a Type 2 cycle (see
(\ref{eqn:B2PeNType2c1I3}--\ref{eqn:B2PeNType2NI3})) and a path (see
(\ref{eqn:B2PeNPathptI3}--\ref{eqn:B2PeNPathpiNI3})), $\ds r(x,I_3)\le
\max_{a\in S_x}\Pi_1(a)+0.5<\min_{b\in S_y}\Pi_1(b)\le l(y,I_3)$ and
therefore, $(x,y)\notin E(I_3)$.

Now we consider the case $S_x=S_y=S$. Let $S=s_1s_2\ldots s_t$,
such that $\Pi_1(s_{i+1})=\Pi_1(s_i)+1$, for $1\le i\le t$. Since
$x$ and $y$ are not adjacent in $S$, by the definition of $\Pi_1$
(Definition \ref{def:Pi1}, Points 1 and 2 for paths and Type 2 cycles
respectively), $|\Pi_1(x)-\Pi_1(y)|>1$. Suppose $S$ is a Type 2 cycle. If
neither $x$ nor $y$ is $s_t$, then, by (\ref{eqn:B2PeNType2c1I3})
and (\ref{eqn:B2PeNType2NI3}), $f(x,I_3)=[\Pi_1(x),\Pi_1(x)+1]$ and
$f(y,I_3)=[\Pi_1(y),\Pi_1(y)+1]$. Since $|\Pi_1(x)-\Pi_1(y)|>1$,
$f(x,I_3)\cap f(y,I_3)=\varnothing$.  If $x=s_t$, then, in $I_1$,
from (\ref{eqn:B2PeNType2ctI1}), $l(x,I_1)=n+\Pi_1(s_t)$ while,
since $y\ne s_{t-1},s_1$ from (\ref{eqn:B2PeNType2othersNI1}),
$r(y,I_1)=n+\Pi_1(y)+1<n+\Pi_1(s_t)=l(x,I_1)$. Therefore, $(x,y)\notin
E(I_1)$. If $S$ is a path, then, assuming without loss of generality
that $\Pi_1(x)<\Pi_1(y)$, from (\ref{eqn:B2PeNPathptI3})
and (\ref{eqn:B2PeNPathpiNI3}), $r(x,I_3)=\Pi_1(x)+1$ and
$l(y,I_3)=\Pi_1(y)$. Since $|\Pi_1(x)-\Pi_1(y)|>1$, $r(x,I_3)<l(y,I_3)$.
Therefore, $(x,y)\notin E(I_3)$.

\paragraph{\boldmath $y\in\mB_2\cup\mP_e$:} First we will show the
following:
\begin{clm}\label{clm:xyg1}
$|\Pi_1(x)-\Pi_1(y)|>1$.
\end{clm}
\begin{proof}
Suppose $S_x=S_y$, that is, both $x$ and $y$ belong to the same component.
Since $x$ is not adjacent to $y$, from the definition of $\Pi_1$
(Definition \ref{def:Pi1}), $|\Pi_1(x)-\Pi_1(y)|>1$.  Suppose $S_x\ne
S_y$. Let $S_y=s_1s_2\ldots s_t$ such that $\Pi_1(s_{i+1})=\Pi_1(s_i)+1$,
$1\le i<t$. If $|\Pi_1(x)-\Pi_1(y)|=1$, then, $y$ must be either $s_1$
or $s_t$. But, $s_1,s_t\in\mN$ since $S_y$ is either a Type 2 cycle or
a path (by Definition \ref{def:Pi1}). This contradicts the assumption
that $y\in\mB_2\cup\mP_e$.\bqed
\end{proof}

Suppose $\Pi_1(x)<\Pi_1(y)$. From Observation
\ref{obs:BunchI3}.\ref{BunchI3:NLess2n}, $r(x,I_3)\le \Pi_1(x)+1$
and from Observation \ref{obs:BunchI3}.\ref{BunchI3:B2PeLess2n},
$l(y,I_3)=\Pi_1(y)$. We have from Claim \ref{clm:xyg1},
$\Pi_1(x)+1<\Pi_1(y)$ and hence $(x,y)\notin
E(I_3)$. Suppose $\Pi_1(x)>\Pi_1(y)$. From Observation
\ref{obs:BunchI1}.\ref{BunchI1:B2PeType2Path}, $r(y,I_1)=n+\Pi_1(y)+1$
and from Observation \ref{obs:BunchI1}.\ref{BunchI1:NType2Path}, $\ds
l(x,I_1)=n+\Pi_1(x)$. From Claim \ref{clm:xyg1}, $\Pi_1(y)+1<\Pi_1(x)$
and hence $(x,y)\notin E(I_1)$. \qed
\end{proof}

\begin{observation}\label{obs:NnonAdjB2PeNI1I3}
If $x\in\mN$ and $y\in\mB_2\cup\mP_e\cup\mN$ such that
$(x,y)\notin E(G)$, then, $(x,y)\notin E(I_1\cap I_3)$. The
proof is as follows: If one of $x$ and $y$ belongs to a Type
1 cycle and the other belongs to a Type 2 cycle or path, from
Observation \ref{obs:BunchI1}.\ref{BunchI1:Type1SepType2PathI1},
$(x,y)\notin E(I_1)$. If both $x$ and $y$ belong to Type 1 cycles,
then, by Lemma \ref{lem:Type1NI1I3}, $(x,y)\notin E(I_1\cap I_3)$.
If both $x$ and $y$ belong to Type 2 cycles and paths, then, by Lemma
\ref{lem:Type2PathNI1I3}, $(x,y)\notin E(I_1\cap I_3)$.
\end{observation}
}
\subsubsection{Vertices of $\mB_1\cup\mC\cup\mP_{2e}$} 
Let $v\in\mB_1\cup\mC\cup\mP_{2e}$. By Table \ref{tab:uniqNeighbor}
(rows 2 and 3), it follows that $v$ is adjacent to exactly one vertex
in $\mN$; let $v'$ be this vertex.
\begin{eqnarray}
&&\textrm{If $v\in\mC\cup\mP_{2e}$, then, } f(v,I_3)=\left[r(v',I_3),n+1\right],\label{eqn:CP2eI3}\\
&&\textrm{If $v\in\mB_1$, then, } f(v,I_3)=\left[r(v',I_3),n\right],\label{eqn:B1I3}
\end{eqnarray}
Note that $v'$ is already assigned an interval in Section \ref{sec:B2PeNI3}.

\subsubsection{Vertices of ${\mP_{2i}}$} 
Let $v\in\mP_{2i}$. By Table \ref{tab:uniqNeighbor} (row 4), $v$ has a
unique neighbor in $\mR\cup\mN$; let $v'$ be this vertex.
\begin{eqnarray}
&&\textrm{if $v'\in\mR$, then, } f(v,I_3)=\left[n+1,n+1\right], \label{eqn:P2iRI3}\\
&&\textrm{if $v'\in\mN$, then, } f(v,I_3)=\left[r(v',I_3),n+1\right].\label{eqn:P2iNI3}
\end{eqnarray}
\journ{
\begin{lemma}\label{lem:BP2iI3}
Let $x\in\mP_{2i}$ and $y\in\mB$. $(x,y)\notin E(I_2\cap I_3)$.
\end{lemma}
\begin{proof}
Recall the notation $\beta(\cdot)$ introduced in Definition
\ref{def:baseVertex}. Let $y\in\Gamma(z)$ for some $z\in\mR$. If
$x\ne\beta(z)$, then by Lemma \ref{lem:RBI2} (Point 1), $(x,y)\notin
E(I_2)$. If $x=\beta(z)$, then, it implies that $x$ is adjacent to $z$ and
by Table \ref{tab:uniqNeighbor} (row 4), $z$ is its only neighbor 
in $\mR\cup\mN$. Since $z\in\mR$, the interval assigned to $x$ is given
in (\ref{eqn:P2iRI3}), from which $l(x,I_3)=n+1$. If $y\in\mB_1$, then
by (\ref{eqn:B1I3}), $r(y,I_3)=n$. If $y\in\mB_2$, from Observation
\ref{obs:BunchI3}.\ref{BunchI3:B2I3}, $r(y,I_3)=n$. Therefore,
$(x,y)\notin E(I_3)$. \qed
\end{proof}
}
\subsubsection{Vertices of ${\mR}$} 
\conf{$\forall v\in\mR, f(v,I_3)=[n,n+1]$.}
\journ{
Every vertex is assigned the following interval:
\begin{equation}\label{eqn:RI3}
\forall v\in\mR, f(v,I_3)=[n,n+1].
\end{equation}
}
\journ{
\begin{lemma}\label{lem:B1CP2eRP2iSuperI3}
$I_3[\mB_1\cup\mC\cup\mP_{2e}\cup\mP_{2i}\cup\mR]$ is a supergraph of
$G[\mB_1\cup\mC\cup\mP_{2e}\cup\mP_{2i}\cup\mR]$.
\end{lemma}
\begin{proof}
First we prove the following:
\begin{clm}\label{clm:nnp1}
For any $x\in\mC\cup\mP_{2e}\cup\mR$, $[n,n+1]\subseteq f(x,I_3)$.
\end{clm}
\begin{proof}
For any $z\in\mN$, from Observation
\ref{obs:BunchI3}.\ref{BunchI3:NLess2n}, $r(z,I_3)\le \Pi_1(z)+1<n$. From
(\ref{eqn:CP2eI3}), for every $x\in\mC\cup\mP_{2e}$, $l(x,I_3)=r(z,I_3)$
for some $z\in\mN$. Since $r(x,I_3)=n+1$, it implies that $[n,n+1]\subset
f(x,I_3)$. If $x\in\mR$, by (\ref{eqn:RI3}), $[n,n+1]=f(x,I_3)$. Hence
proved. \bqed
\end{proof}

To prove the lemma, we need to show that if
$(x,y)\in E(G)$, then, $(x,y)\in E(I_3)$, where
$x,y\in\mB_1\cup\mC\cup\mP_{2e}\cup\mP_{2i}\cup\mR$. For any
$x\in\mP_{2i}$, from (\ref{eqn:P2iRI3}--\ref{eqn:P2iNI3}),
$r(x,I_3)=n+1$ and using Claim \ref{clm:nnp1}, we infer that
$I_3[\mC\cup\mP_{2e}\cup\mP_{2i}\cup\mR]$ is a clique. For any
$x\in\mB_1$, from (\ref{eqn:B1I3}), $r(x,I_3)=n$ and again by using Claim
\ref{clm:nnp1}, we observe that $I_3[\mC\cup\mP_{2e}\cup\mB_1\cup\mR]$ is
a clique. Therefore, the only case we have to consider is $x\in\mP_{2i}$
and $y\in\mB_1$. However, from Table \ref{tab:nonAdj} (row 6), this case
is not possible. Hence proved.  \qed
\end{proof}

\begin{lemma}\label{lem:VminusASuperI3}
$I_3[V\setminus\mA]$ is a supergraph of $G[V\setminus\mA]$.
\end{lemma}
\begin{proof}
Let $x,y\in V\setminus\mA$ be two adjacent vertices in
$G$. Note that $V\setminus\mA=(\mB_2\cup\mP_e\cup\mN)\cup
(\mB_1\cup\mC\cup\mP_{2e}\cup\mP_{2i}\cup\mR)$.
If $x,y\in\mB_2\cup\mP_e\cup\mN$, by Observation
\ref{obs:B2PeNSuperI3}, $x$ and $y$ are adjacent in $I_3$. If
$x,y\in\mB_1\cup\mC\cup\mP_{2e}\cup\mP_{2i}\cup\mR$, then by
Lemma \ref{lem:B1CP2eRP2iSuperI3}, $x$ and $y$ are adjacent in
$I_3$. The remaining case is $x\in\mB_2\cup\mP_e\cup\mN$ and
$y\in\mB_1\cup\mC\cup\mP_{2e}\cup\mP_{2i}\cup\mR$.

Let $x\in\mB_2$. From Table \ref{tab:nonAdj} (row 6) we note that
$y\notin\mC\cup\mP_{2e}\cup\mP_{2i}$. Therefore,
$y\in\mB_1\cup\mR$. If $y\in\mB_1$, by (\ref{eqn:B1I3}), $r(y,I_3)=n$ and
if $y\in\mR$, by (\ref{eqn:RI3}), $l(y,I_3)=n$. By Observation
\ref{obs:BunchI3}.\ref{BunchI3:B2I3}, $r(x,I_3)=n$ and therefore, $x$ is
adjacent to $y$ in $I_3$.

Let $x\in\mP_e$. From Table \ref{tab:nonAdj} (rows 6, 4, 8 and 5
respectively) $y\notin\mB_1\cup\mC\cup\mP_{2i}\cup\mR$. Therefore,
$y\in\mP_{2e}$. By (\ref{eqn:CP2eI3}), $r(y,I_3)=n+1$ and by Observation
\ref{obs:BunchI3}.\ref{BunchI3:PeI3}, $r(x,I_3)=n+1$. Hence, $x$
is adjacent to $y$ in $I_3$.

Let $x\in\mN$.  From Table \ref{tab:nonAdj} (row 2), $y\notin\mR$. This
implies that $y\in\mB_1\cup\mC\cup\mP_{2e}\cup\mP_{2i}$. From Table
\ref{tab:uniqNeighbor} (rows 2, 3 and 4), $x$ is the unique neighbor of $y$
in $\mN$. From (\ref{eqn:CP2eI3}), (\ref{eqn:B1I3}) and (\ref{eqn:P2iNI3}),
$l(y,I_3)=r(x,I_3)$. Hence, $x$ is adjacent to $y$ in $I_3$.  \qed
\end{proof}

\begin{lemma}\label{lem:Nresolved}
If $x\in\mN$ and $y\in V\setminus\mA$ such that $(x,y)\notin E(G)$, then,
$(x,y)\notin E(I_1\cap I_3)$.
\end{lemma}
\begin{proof}
If $y\in\mB_2\cup\mP_e\cup\mN$, then by Observation
\ref{obs:NnonAdjB2PeNI1I3}, $(x,y)\notin E(I_1\cap I_3)$.  If $y\in\mR$,
then by (\ref{eqn:RI3}), $l(y,I_3)=n$. If $y\in\mP_{2i}$ and is
not adjacent to any vertex in $\mN$, then by (\ref{eqn:P2iRI3}),
$l(y,I_3)=n+1$. By Observation \ref{obs:BunchI3}.\ref{BunchI3:NLess2n},
$r(x,I_3)<n$ and therefore, in both the cases, $(x,y)\notin E(I_3)$.
Now we consider the remaining cases (1) $y\in\mP_{2i}$ such that $y$ has
a neighbor in $\mN$ and (2) $y\in\mB_1\cup\mC\cup\mP_{2e}$. From Table
\ref{tab:uniqNeighbor} (rows 4, 2 and 3 respectively), in each case,
$y$ has exactly one neighbor in $\mN$ and let this vertex be $z$. Since
$x\ne z$, either $\Pi_1(x)<\Pi_1(z)$ or $\Pi_1(z)<\Pi_1(x)$.

Suppose $\Pi_1(x)<\Pi_1(z)$. In $I_3$, from the interval assignments
for vertices of $\mC\cup\mP_{2e}$, $\mB_1$ and $\mP_{2i}$ in
(\ref{eqn:CP2eI3}), (\ref{eqn:B1I3}) and (\ref{eqn:P2iNI3})
respectively, $l(y,I_3)=r(z,I_3)$. From Observation
\ref{obs:BunchI3}.\ref{BunchI3:NLess2n}, $r(x,I_3)\le
\Pi_1(x)+1<\Pi_1(z)+0.5\le r(z,I_3)$. Hence, $(x,y)\notin E(I_3)$.

Suppose $\Pi_1(x)>\Pi_1(z)$. In $I_1$, from the interval
assignments for vertices of $\mB_1$, $\mP_{2e}$ and $\mP_{2i}$ in
(\ref{eqn:B1I1}), (\ref{eqn:P2eI1}) and (\ref{eqn:P2iNI1}) respectively,
$r(y,I_1)=l(z,I_1)$. For $\mC$ in (\ref{eqn:CI1}), $r(y,I_1)\le
l(z,I_1)+0.5$. From Observation \ref{obs:BunchI1}.\ref{BunchI1:xLessy},
$l(x,I_1)\ge l(z,I_1)+1$ and therefore, $(x,y)\notin E(I_1)$. Hence
proved.\qed
\end{proof}

\begin{lemma}\label{lem:NTouch}
Let $x\in\mN$ and $y\in V\setminus\mA$ such that $(x,y)\in E(G)$. Then,
for some $I\in\{I_1,I_3\}$, either $l(x,I)=r(y,I)$ or $l(y,I)=r(x,I)$.
\end{lemma}
\begin{proof}
Note that $V\setminus\mA=(\mB_2\cup\mP_e\cup\mN)\cup
(\mB_1\cup\mC\cup\mP_{2e}\cup\mP_{2i}\cup\mR)$. Suppose
$y\in\mB_2\cup\mP_e\cup\mN$. Since $x$ is adjacent to $y$, they
belong to the same component in $G[\mB_2\cup\mP_e\cup\mN]$, which
is either a Type 1 cycle, Type 2 cycle or a path. Let this component
be $S=s_1s_2\ldots s_t$, where $\Pi_1(s_{i+1})=\Pi_1(s_i)+1$, $1\le
i<t$. Suppose $S$ is a Type 1 cycle.  If neither $x$ nor $y$ is $c_1$,
then by Observation \ref{obs:Type1I1}.\ref{Type1I1:SMinusc1Touch},
this condition is satisfied in $I_1$. If $x$ or $y$ is $c_1$, then by
Observation \ref{obs:Type1I3}.\ref{Type1I3:touch}, this is satisfied
in $I_3$. Suppose $S$ is a Type 2 cycle. If $x$ or $y$ is $c_1$, say
$x=c_1$, then by Observation \ref{obs:Type2I3}.\ref{Type2I3:c1Touch},
$r(x,I_3)=l(y,I_3)$. Now we will assume that neither
$x$ nor $y$ is $c_1$. Suppose $\Pi_1(x)>\Pi_1(y)$.
By Observation \ref{obs:Type2I1}.\ref{Type2I1:touch},
$l(x,I_1)=r(y,I_1)$. Moreover, since $c_t\in\mN$, this holds for
the case $x=c_t$ and $y=c_{t-1}$ too. Therefore, we can assume that
$x,y\notin\{c_1,c_t\}$. If $\Pi_1(x)<\Pi_1(y)$, then, by Observation
\ref{obs:Type2I3}.\ref{Type2I3:touch}, $l(y,I_3)=r(x,I_3)$. Finally,
suppose $S$ is a path. If $y\in\mN$, then, without loss of generality
we can assume that $\Pi_1(x)<\Pi_1(y)$ and therefore, by Observation
\ref{obs:PathI3}.\ref{PathI3:touch}, the condition is satisfied in
$I_3$.  If $y\in\mB_2\cup\mP_e$ and $\Pi_1(x)<\Pi_1(y)$, then, again by
Observation \ref{obs:PathI3}.\ref{PathI3:touch}, $r(x,I_3)=l(y,I_3)$ and
if $\Pi_1(x)>\Pi_1(y)$, by Observation \ref{obs:PathI1}.\ref{PathI1:touch}
$r(y,I_3)=l(x,I_3)$.

Suppose $y\in\mB_1\cup\mC\cup\mP_{2e}\cup\mP_{2i}$. By Table
\ref{tab:uniqNeighbor} (rows 2, 3 and 4), $x$ is the unique neighbor
of $y$ in $\mN$. From (\ref{eqn:CP2eI3}), (\ref{eqn:B1I3}) and
(\ref{eqn:P2iNI3}), it follows that $l(y,I_3)=r(x,I_3)$. By Table
\ref{tab:nonAdj} (row 2), $x$ is not adjacent to any vertex in $\mR$.
Thus, we have covered all possible cases for $y\in V\setminus\mA$. Hence,
proved.\qed
\end{proof}

\begin{observation}\label{obs:RBTouch}
If $x\in\mR$ and $y\in\mB$ are adjacent in $G$, then,
$l(x,I_3)=r(y,I_3)=n$. This follows from the fact that
$\mB=\mB_1\cup\mB_2$, (\ref{eqn:B1I3}), (\ref{eqn:RI3}) and Observation
\ref{obs:BunchI3}.\ref{BunchI3:B2I3}.
\end{observation}
}
\subsubsection{Vertices of $\mA$}
\conf{$\forall v\in\mA, f(v,I_3)=[1,n+1]$.}
\journ{
Every vertex is assigned the following interval:
\begin{equation}\label{eqn:AI3}
\forall v\in\mA, f(v,I_3)=[1,n+1].
\end{equation}
}
\journ{
\begin{lemma}\label{lem:superI3}
$I_3$ is a supergraph of $G$.
\end{lemma}
\begin{proof}
By Lemma \ref{lem:VminusASuperI3}, $I_3[V\setminus\mA]$ is a supergraph of
$G[V\setminus\mA]$. From the interval assignments in $I_3$ for vertices in
$V\setminus\mA$, it is easy to infer that if $x\in V\setminus\mA$, then,
$f(x,I_3)\subset [1,n+1]$. Since by (\ref{eqn:AI3}), for any $y\in\mA$,
$f(y,I_3)=[1,n+1]$, it follows that $y$ is adjacent to $x$. Clearly,
$I_3[\mA]$ is a clique. Therefore, $I_3$ is a supergraph of $G$.
\qed
\end{proof}

\subsection{Proof of $E(G)=E(I_1)\cap E(I_2)\cap E(I_3)$}\label{sec:verify}
We will prove Theorem \ref{thm:mainTheorem} by showing that
$E(G)=E(I_1)\cap E(I_2)\cap E(I_3)$. We have already proved that $I_1$,
$I_2$ and $I_3$ are supergraphs of $G$ (in Lemmas \ref{lem:superI1},
\ref{lem:superI2} and \ref{lem:superI3}, respectively). In this
section, we will show the following: If two vertices $s$ and $t$
are not adjacent in $G$, then there exists at least one interval
graph $I\in\{I_1,I_2,I_3\}$, such that $(s,t)\notin I$.  Recall the
partitioning of $V$ illustrated in Figure \ref{fig:cubicPartTree}:
$V=\mA\cup\mN\cup\mC\cup\mR\cup\mB\cup\mP$. Now we will consider one
by one all possible cases and in each case show that $s$ and $t$
are not adjacent in at least one of the interval graphs.

\paragraph{\boldmath $s\in\mA$, $t\in V$:}
If $t\in\mA$, then by Lemma \ref{lem:AGI1}, $(s,t)\notin E(I_1)$. If
$t\in\mN$, then by Lemma \ref{lem:ANnonAdj}, $(s,t)\notin
E(I_1\cap I_2)$ and if $t\in V\setminus(\mA\cup\mN)$, then by Lemma
\ref{lem:AANnonAdj} $(s,t)\notin E(I_1)$.

\paragraph{\boldmath $s\in\mN$, $t\in V\setminus\mA$:}
By Lemma \ref{lem:Nresolved}, $(s,t)\notin E(I_1\cap I_3)$.

\paragraph{\boldmath $s\in\mC$, $t\in
V\setminus(\mA\cup\mN)=\mC\cup\mP\cup\mR\cup\mB$:} If $t\in
\mC\cup\mP$, by Lemma \ref{lem:CPI1I2}, $(s,t)\notin E(I_1\cap I_2)$. If
$t\in\mR\cup\mB$, by Observation \ref{obs:RBCP}.\ref{RBCP:RBCPeP2eI2},
$(s,t)\notin E(I_2)$.

\paragraph{\boldmath $s\in\mP$, $t\in
V\setminus(\mA\cup\mN\cup\mC)=\mP\cup\mR\cup\mB$:} If $t\in\mP$,
by Lemma \ref{lem:CPI1I2}, $s$ and $t$ are not adjacent in either
$I_1$ or $I_2$. Let $t\in\mR\cup\mB$. If $s\in\mP_e\cup\mP_{2e}$,
then by Observation \ref{obs:RBCP}.\ref{RBCP:RBCPeP2eI2}, $(s,t)\notin
E(I_2)$.  Finally, let $s\in\mP_{2i}$. If $t\in\mR$, then by Observation
\ref{obs:RBCP}.\ref{RBCP:RP2iI2} $(s,t)\notin E(I_2)$ and if $t\in\mB$,
by Lemma \ref{lem:BP2iI3}, $(s,t)\notin E(I_2\cap I_3)$.

\paragraph{\boldmath $s,t\in\mR\cup\mB$:} By Lemma \ref{lem:RBI2}
(Statement 2), $(s,t)\notin E(I_2)$.

\subsection{Proof of Theorem \ref{thm:mainTheorem}}\label{sec:boundary}
We have proved that $G$ has a $3$-box representation. Now we will
show that in this $3$-box representation, any two intersecting boxes
intersect only at their boundaries and hence complete the proof of Theorem
\ref{thm:mainTheorem}. For this to happen, the following condition needs
to be satisfied:
\begin{condition}\label{con:touch}
Let $s$ and $t$ be adjacent in $G$. For some $I\in\{I_1,I_2,I_3\}$,
either $l(s,I)=r(t,I)$ or $l(t,I)=r(s,I)$.
\end{condition}
As in the previous section, we will consider one by one all the possible
cases:

\paragraph{\boldmath $s\in\mA$, $t\in V$:} If $t\in\mA$, then by 
Observation \ref{obs:AATouch} and if $t\in\mN$, by Observation
\ref{obs:BunchI1}.\ref{BunchI1:ANTouch}, Condition \ref{con:touch} is
satisfied. By Table \ref{tab:nonAdj} (rows 1--3), $s$ is not adjacent
to any vertex in $V\setminus(\mA\cup\mN)$.

\paragraph{\boldmath $s\in\mN$, $t\in V\setminus\mA$:} By Lemma
\ref{lem:NTouch}, Condition \ref{con:touch} is satisfied.

\paragraph{\boldmath $s\in\mC$, $t\in
V\setminus(\mA\cup\mN)=\mC\cup\mP\cup\mR\cup\mB$:} By Table
\ref{tab:nonAdj} (rows 4--6), $t\notin\mP\cup\mR\cup\mB$. Therefore, the
only case to be considered is $t\in\mC$. Let $s,t\in C$, where
$C\subseteq\mC$ is a special cycle. Let $C=c_1c_2\ldots c_t$
where $c_1$ is the special vertex. If $x=c_1$, then by Observation
\ref{obs:CI1}.\ref{CI1:special}, Condition \ref{con:touch} is satisfied. If
neither $x$ nor $y$ is $c_1$, then by Observation
\ref{obs:CI2}.\ref{CI2:touch}, Condition \ref{con:touch} is satisfied in
$I_2$.

\paragraph{\boldmath $s\in\mP$, $t\in
V\setminus(\mA\cup\mN\cup\mC)=\mP\cup\mR\cup\mB$:} If $t\in\mP$,
then by Observation \ref{obs:PSuperI2}, Condition \ref{con:touch}
is met in $I_2$. If $t\in\mR$, then from Table \ref{tab:nonAdj} (row
5) we infer that $s\in\mP_{2i}$. By Observation \ref{obs:P2iRTouch},
Condition \ref{con:touch} is met in $I_1$. By Table \ref{tab:nonAdj}
(row 6), $t\notin\mB$.

\paragraph{\boldmath $s\in\mR$, $t\in
V\setminus(\mA\cup\mN\cup\mC\cup\mP)=\mR\cup\mB$:} By Table
\ref{tab:nonAdj} (row 10) $t\notin\mR$. By Observation \ref{obs:RBTouch},
Condition \ref{con:touch} is satisfied in $I_3$.

\paragraph{\boldmath $s,t\in\mB$:} Condition \ref{con:touch} is satisfied
in $I_2$ (See Figure \ref{fig:componentsYInterval}).

Thus, we have completed the proof for cubic graphs. Now, applying
Lemma \ref{lem:cubicEnough}, it follows that any graph with maximum
degree $3$ has a $3$-box representation. Thus, we have proved Theorem
\ref{thm:mainTheorem}.
}
\confExt{
\section{Algorithmic aspects}
Now we briefly explain how our construction of the $3$-box representation
can be realized in $O(n)$ time, where $n$ is the number of vertices
in the graph. Firstly, we note that the process can be split into
three stages: (1) Partitioning the vertex set as illustrated in
Figure \ref{fig:cubicPartTree}, (2) ordering the vertices of $\mA$,
$\mB_2\cup\mP_e\cup\mN$ and $\mC\cup\mP$ according to Definitions
\ref{def:PiA}, \ref{def:Pi1} and \ref{def:Pi2} and finally, (3)
assigning intervals to all the vertices.  In Stage (1) the first
step is to extract special cycles and special paths as described in
Algorithm \ref{alg:primPart}. This is the only non-trivial part of the
construction and we analyze its complexity in the appendix. We will show
that a special cycle or path can be extracted from a graph with maximum
degree $3$ in time linear to the number of vertices in it. This will
imply that Algorithm \ref{alg:primPart} can be implemented in $O(n)$
time. Algorithm \ref{alg:finePart} takes $O(n)$ time since in every
iteration we need to only check if a vertex in $\mN_1$ has two neighbors
in $\mA_1$ in that iteration and accordingly move or retain the vertex
and its neighbors. It is easy to see that the finer partitioning of
$\mP$, $\mN$ and $\mB$ can be accomplished in linear time.  Stage (2)
involves ordering the vertices of sets $\mA$, $\mB_2\cup\mP_e\cup\mN$
and $\mC\cup\mP$ component wise. Since each of these sets induce a graph
of maximum degree $2$, they can be ordered in linear time. Stage (3)
only involves assignment of intervals to the vertices and can be achieved
in linear time.
}

\section{Conclusion}
We showed that every graph of maximum degree $3$ has a $3$-box
representation and therefore, its boxicity is at most $3$. One interesting
question is whether we can characterize cubic graphs which have a $2$-box
representation. Answering this will also determine if the boxicity of a
cubic graph can be computed in polynomial time. One could also try to
extend the proof techniques used in this paper to graphs with maximum
degree $4$ and $5$ in order to improve the bounds on their boxicity.

\bibliography{docsdb}

\appendix
\section{Appendix}
Let $G$ be a graph with maximum degree $3$.

\subsection{Extending a non-special induced cycle or path}\label{sec:extend} 
Suppose $C$ is a non-special induced cycle in $G$. Then, by Definition
\ref{def:specialCycle}, it follows that there exists a vertex $x\in
C$, such that $C\setminus\{x\}$ belongs to a cycle or path of length
$|C|+1$. We call such a vertex a {\bf removable vertex}. Extending
a non-special induced cycle $C$ corresponds to removing a removable
vertex $x$ and adding two new vertices $a$ and $b$ such that
$(C\setminus\{x\})\cup\{a,b\}$ is an induced cycle or a path. There
are only two possible ways in which a non-special induced cycle can
be extended and these are illustrated in Figure \ref{fig:cycle2Path}.
The vertices which are added to or removed from $C$ in an extension
operation are called {\bf participating vertices}. In the figure, $x$,
$a$ and $b$ are the participating vertices.

\begin{figure}[bh]
\begin{center}
\epsfig{file=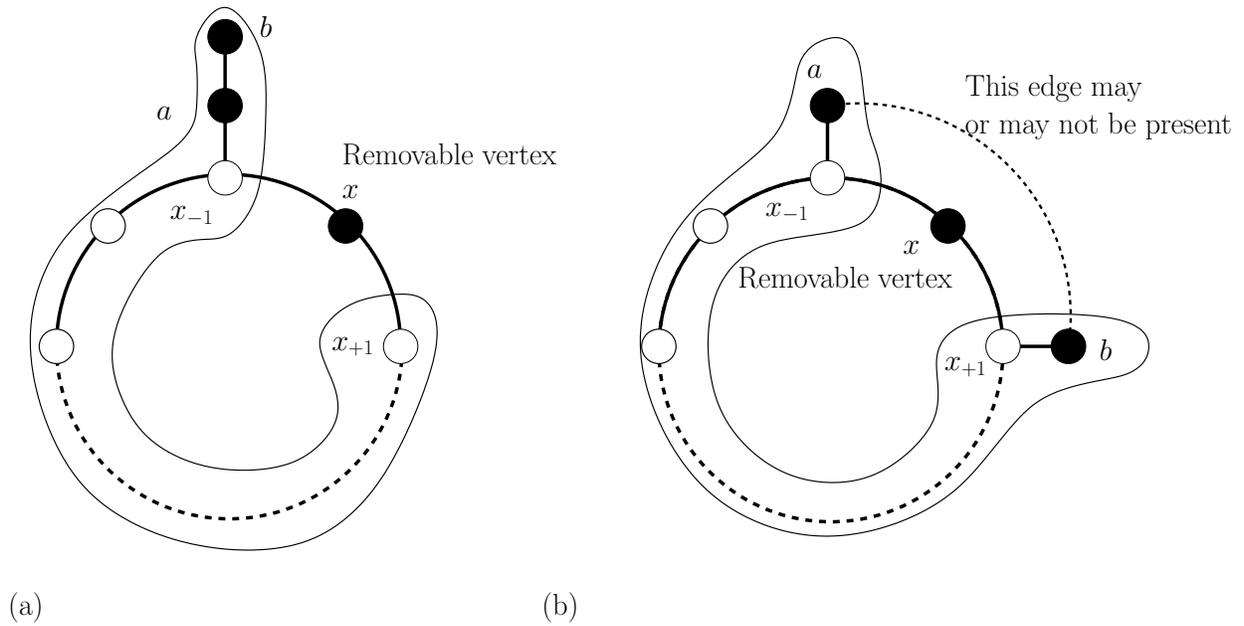,width=16.5cm}
\end{center}
\caption{The possible ways in which a non-special induced cycle can be
extended. The vertices marked black are the participating vertices.
\label{fig:cycle2Path}}
\end{figure}

\begin{lemma}\label{lem:cycleExt}
Suppose $C$ is a non-special induced cycle and $x\in C$. It takes
constant time to verify whether $x$ is a removable vertex or not. If
$x$ is a removable vertex, then, the extension of $C$ by removing $x$
can again be achieved in constant time.
\end{lemma}
\begin{proof}
Consider the possible ways in which $C$ can be extended as shown in
Figure \ref{fig:cycle2Path}. To verify if $x$ is a removable vertex,
we need to only check if the vertices $a$ and $b$ exist. Similarly,
given that $x$ is a removable vertex, we need to only find $a$
and $b$ to extend $C$ by removing $x$. Recalling that $\Delta(G)\le3$,
it is easy to see that this can be done in constant time.\qed
\end{proof}

Let $P$ be a non-special induced path in $G$. By Definition
\ref{def:specialPath}, it implies that either:
\begin{enumerate}
\item it is not maximal in the sense that it is part of an induced
cycle or a longer induced path, or
\item for some end point of $P$, say $x$, $P\setminus\{x\}$ belongs to an
induced cycle of size $\ge|P|$ or an induced path of length $\ge|P|+1$. We
call $x$ a {\bf removable end point} of $P$.
\end{enumerate}
Extending a non-special path $P$ corresponds to the following operations:
\begin{enumerate}
\item If $P$ is not maximal, then, we add a new vertex $y$ to $P$ such
that $P\cup\{y\}$ is an induced path or cycle. Clearly, $y$ must be
a neighbor of an end point of $P$ such that it is not adjacent to any
interior vertex of $P$.
\item If $P$ is maximal, then, we remove a removable end point $x$ and
either 
\begin{enumerate}[(a)]
\item add a single new vertex $a$ such that $(P\setminus\{x\})\cup\{a\}$
is an induced cycle (note that in this case $a$ has to be adjacent to the
neighbor of $x$ in $P$) OR
\item add two new vertices $a$ and $b$ such that
$(P\setminus\{x\})\cup\{a,b\}$ is an induced cycle or a path (in this
case, $a$ and $b$ are adjacent and $a$ is adjacent to the neighbor
of $x$ in $P$).
\end{enumerate}
This is illustrated in Figure \ref{fig:path2Cycle}.
\end{enumerate}
As in the case of extending a cycle, the vertices which are added to or
removed from $P$ are called {\bf participating vertices}. In case 1,
$y$ is the participating vertex. In case 2(a), $x$ and $a$ are
participating vertices and in case 2(b), $x$, $a$ and $b$ are participating
vertices.

\begin{figure}[th]
\begin{center}
\epsfig{file=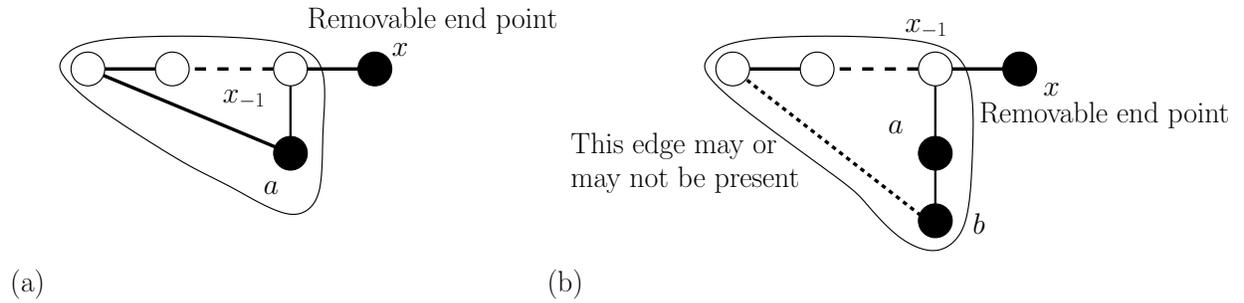,width=16.5cm}
\end{center}
\caption{The ways in which a non-special induced path can be extended. The
vertices marked black are the participating vertices.
\label{fig:path2Cycle}}
\end{figure}

\begin{lemma}\label{lem:pathExt}
If $P$ is an induced path, then, in constant time, it can be verified
whether it is a special path or not. If not, then, in constant time it
can be extended.
\end{lemma}
\begin{proof}
First we need to check if $P$ is maximal or not, that is, whether it
is part of a larger induced cycle or a longer induced path. This can
be done in constant time. If $P$ is maximal, then, we need to check if
there is a removable end point and then extend $P$ by removing it. For
this, as shown in Figure \ref{fig:path2Cycle}, we need to check if the
vertices $x$, $a$ and $b$ exist. Since $\Delta(G)\le3$, it is easy to
see that this can be done in constant time. It is also trivial to verify
that the extension can be achieved in constant time. Hence proved. \qed
\end{proof}

\subsection{An algorithm to find a special cycle or path}
We now give an iterative algorithm to obtain a special cycle or path. The
outline of the algorithm is as follows: Let $S$ be the set which holds
the vertices of the special cycle or path at the termination of the
algorithm. We start with $S$ containing an arbitrary vertex. In each
iteration, we extend it as described in Section \ref{sec:extend}. The
algorithm terminates when $S$ induces a special cycle or path.
In Algorithm \ref{alg:special}, we present an outline of this procedure.

\begin{algorithm}[h]
\caption{\label{alg:special}}
\SetKwInOut{Input}{input}
\SetKwInOut{Output}{output}
\SetKw{Sf}{specialFlag}
\Input{Graph $G$ with maximum degree $3$}
\Output{Set $S\subseteq V(G)$ which induces a special cycle or path in $G$}
Let $S=\{x\}$ where $x$ is an arbitrary vertex\; 
Let $R=\varnothing$, the set of potential removable vertices of $S$\;
Let \Sf $=0$; \tcp{If set to $1$, it implies that $S$ is a special
cycle or path}
\While{\Sf$=0$}{
   \uIf{$S$ induces a cycle}{
      \uIf{$R=\varnothing$}{
         \Sf$=1$; \tcp{No removable vertices and therefore, $S$
         is a special cycle (see Observation \ref{obs:pot})}
      }
      \Else{
         Choose any vertex $x$ from $R$\;\nllabel{lin:xR}
         Extend $S$ by removing $x$ as described in Section \ref{sec:extend}\;\nllabel{lin:extendCycle}
      }
   }
   \ElseIf{$S$ induces a path}{
      \uIf{$S$ is a special path}{
         \Sf$=1$\;
      }
      \Else{
         Extend $S$ as described in Section \ref{sec:extend}\;\nllabel{lin:extendPath}
      }
   }
   Update $R$ as described in Section \ref{sec:update}\;\nllabel{lin:updateR}
}
\end{algorithm}

\subsubsection{Potential removable vertices}\label{sec:update}
From Lemma \ref{lem:pathExt}, we note that in constant time we can
recognize a non-special induced path and extend it. However, in view of
Lemma \ref{lem:cycleExt}, to recognize and extend a non-special induced
cycle in constant time, we first need a strategy to find a removable
vertex. For efficiently finding removable vertices, we maintain a list
of potential removable vertices which is updated in each iteration.

\begin{definition}{\bf Potential removable vertices:}\label{def:potRemove}
A vertex $x\in S$ is a potential removable vertex if it has two neighbors
in $S$, say $x_{-1}$ and $x_{+1}$ and satisfies at least one of the following
conditions: There are two vertices $a$ and $b$ such that
\begin{enumerate}
\item $a$ is adjacent to only $x_{-1}$ (or $x_{+1}$) in $S\setminus\{x\}$
and $b$ is adjacent to $a$ and not any vertex in $S\setminus\{x\}$. This
corresponds to case (a) in Figure \ref{fig:cycle2Path}; OR
\item $a$ is adjacent to $x_{-1}$ in $S\setminus\{x\}$ and $b$ is adjacent
to only $x_{+1}$ in $S\setminus\{x\}$. This corresponds to case (b)
in Figure \ref{fig:cycle2Path}.
\end{enumerate}
\end{definition}

From the definition of removable vertices at the beginning of Section
\ref{sec:extend}, we infer the following:
\begin{observation}\label{obs:pot}
If $S$ induces a cycle, then, all potential removable vertices in $S$
will correspond to removable vertices. If there are no potential removable
vertices, then, it implies that $S$ is a special cycle.
\end{observation}

\begin{lemma}\label{lem:checkRemovable}
In constant time, we can check if a particular vertex in $S$ is a potential
removable vertex.
\end{lemma}
The proof is similar to that of Lemma \ref{lem:cycleExt}.

In the algorithm, we maintain a set of all potential removable vertices,
which we denote as $R$. From Observation \ref{obs:pot}, it follows that if $S$
induces a cycle, we can decide whether it is special or not by just
checking if $R$ is empty or not. Therefore, given an induced cycle and the
corresponding $R$, we can recognize in constant time whether it is a special
cycle or not. Now, we show that after each extension of $S$, $R$ can be
updated in constant time.
\begin{lemma}\label{lem:updateR}
In each iteration of Algorithm \ref{alg:special}, the set of potential
removable vertices, $R$ can be updated in constant time.
\end{lemma}
\begin{proof}
Let us consider a vertex $x\in R$ before extension of $S$. Let $X$ denote
the set containing $x$ and the associated vertices $x_{-1}$, $x_{+1}$,
$a$ and $b$ (see Definition \ref{def:potRemove}). Since $S$ is extended,
there are participating vertices. We observe that $x$ will remain as a
potential removable vertex if no vertex in $X$ and no neighbor of $X$ is
a participating vertex. This implies that if a vertex is at a distance
$5$ or more from any participating vertices, then, clearly its status
as a potential removable vertex or not remains unchanged. Therefore,
we need to only check all the vertices at a distance $4$ or less from
each participating vertex. The number of such vertices is a constant
since $\Delta(G)=3$ and by Lemma \ref{lem:checkRemovable} verifying for
each vertex takes only constant time. Hence proved.  \qed
\end{proof}

\begin{lemma}\label{lem:2in1}
If the special cycle or path extracted by Algorithm \ref{alg:special}
is of size $l$, then, the total number of iterations required is at
most $2l+2$.
\end{lemma}
\begin{proof}
We will prove the lemma by showing that in Algorithm \ref{alg:special},
for every two iterations (excluding the last two) the size of $S$
increases by at least $1$. If $S$ induces a cycle at the beginning of the 
$i$th iteration, then from Section \ref{sec:extend}, it follows that
$S$ is extended to a cycle or path of size $|S|+1$ at the end of
the iteration. If $S$ induces a path at the beginning of the $i$th
iteration, either $S$ is extended to a cycle or path of size $|S|+1$
or to a cycle of size $|S|$ at the end of the iteration. In the latter
case, assuming $S$ is not a special cycle in the $(i+1)$th iteration,
it is extended to a cycle or path of size $|S|+1$ at the end of the
$(i+1)$th iteration. Hence, proved.  \qed
\end{proof}

From Lemmas \ref{lem:cycleExt}--\ref{lem:updateR}, it follows that Lines
\ref{lin:extendCycle}, \ref{lin:extendPath} and
\ref{lin:updateR} in Algorithm \ref{alg:special} require constant time. The 
sets $S$ and $R$ may be implemented as doubly linked
lists and with each vertex we can associate membership flags and pointers
to its place in each of the linked lists. Given this setup, each iteration
requires constant time. From Lemma \ref{lem:2in1},
the number of iterations is bounded by $2l+2$. Therefore, the algorithm
takes $O(l)$ time to terminate. Since the total number of vertices in
the set of special cycles and paths will be bounded above by $n$, the
overall running time of Algorithm \ref{alg:primPart} is $O(n)$.

\end{document}